\documentclass[reqno,a4paper]{amsart}
\usepackage{amssymb}
\usepackage{amsmath}
\newcommand{\angles}[1]{\langle #1 \rangle}
\newcommand{\half}{\frac{1}{2}}
\newcommand{\abs}[1]{\vert #1 \vert}
\newcommand{\norm}[1]{\left\Vert #1 \right\Vert}
\newcommand{\R}{\mathbb{R}}

\begin{document} 
\newtheorem{prop}{Proposition}[section]
\newtheorem{Def}{Definition}[section]
\newtheorem{theorem}{Theorem}[section]
\newtheorem{lemma}{Lemma}[section]
 \newtheorem{Cor}{Corollary}[section]

\title[Yang-Mills in Lorenz gauge]{\bf Low regularity local well-posedness for the Yang-Mills equation in Lorenz gauge}
\author[Hartmut Pecher]{
{\bf Hartmut Pecher}\\
Fakult\"at f\"ur  Mathematik und Naturwissenschaften\\
Bergische Universit\"at Wuppertal\\
Gau{\ss}str.  20\\
42119 Wuppertal\\
Germany\\
e-mail {\tt pecher@math.uni-wuppertal.de}}
\date{}

\begin{abstract}
We prove that the Yang-Mills equation in Lorenz gauge in the (n+1)-dimensional case is locally well-posed for data of the gauge potential in $H^s$ and the curvature in $H^r$ , where $s >\frac{n}{2}-\frac{7}{8}$ , $r > \frac{n}{2}-\frac{7}{4}$ , if $n \ge 4$, and $ s > \frac{3}{4}$ , $ r > - \frac{1}{8}$ , if $n=3$. The proof is based on the fundamental results of Klainerman-Selberg \cite{KS} and on the null structure of most of the nonlinear terms detected by Selberg-Tesfahun \cite{ST} and Tesfahun \cite{Te}.
\end{abstract}
\maketitle
\renewcommand{\thefootnote}{\fnsymbol{footnote}}
\footnotetext{\hspace{-1.5em}{\it 2010 Mathematics Subject Classification:} 
35Q40, 35L70 }
\normalsize 
\setcounter{section}{0}

\section{Introduction}

\noindent 
Let $\mathcal{G}$ be the Lie group $SO(n,\mathbb{R})$ (the group of orthogonal matrices of determinant 1) or $SU(n,\mathbb{C})$ (the group of unitary matrices of determinant 1) and $g$ its Lie algebra $so(n,\mathbb{R})$ (the algebra of trace-free skew symmetric matrices) or $su(n,\mathbb{C})$ (the algebra of trace-free skew hermitian matrices) with Lie bracket $[X,Y] = XY-YX$ (the matrix commutator). 
For given  $A_{\alpha}: \mathbb{R}^{1+n} \rightarrow g $ we define the curvature $F=F[A]$ by
\begin{equation}
\label{curv}
 F_{\alpha \beta} = \partial_{\alpha} A_{\beta} - \partial_{\beta} A_{\alpha} + [A_{\alpha},A_{\beta}] \, , 
\end{equation}
where $\alpha,\beta \in \{0,1,...,n\}$ and $D_{\alpha} = \partial_{\alpha} + [A_{\alpha}, \cdot \,]$ .

Then the Yang-Mills system is given by
\begin{equation}
\label{1}
D^{\alpha} F_{\alpha \beta}  = 0
\end{equation}
in Minkowski space $\mathbb{R}^{1+n} = \mathbb{R}_t \times \mathbb{R}^n_x$ , where $n \ge 3$, with metric $diag(-1,1,...,1)$. Greek indices run over $\{0,1,...,n\}$, Latin indices over $\{1,...,n\}$, and the usual summation convention is used.  
We use the notation $\partial_{\mu} = \frac{\partial}{\partial x_{\mu}}$, where we write $(x^0,x^1,...,x^n)=(t,x^1,...,x^n)$ and also $\partial_0 = \partial_t$.

Setting $\beta =0$ in (\ref{1}) we obtain the Gauss-law constraint
\begin{equation}
\nonumber
\partial^j F_{j 0} + [A^j,F_{j0} ]=0 \, .
\end{equation}

The total energy for YM, at time $t$,  is given by
$$
  \mathcal E(t) =\sum_{0\le \alpha, \beta\le n} \int_{\R^n} \abs{F_{\alpha \beta}(t,x)}^2  \, dx,
$$
and is conserved for a smooth solution decaying sufficiently fast at spatial infinity, 
i.e., $$ \mathcal E(t)= \mathcal E(0).$$

The system is gauge invariant. Given a sufficiently smooth function $U: {\mathbb R}^{1+n} \rightarrow \mathcal{G}$ we define the gauge transformation $T$ by $T A_0 = A_0'$ , 
$T(A_1,...,A_n) = (A_1',...,A_n'),$ where
\begin{align*}
A_{\alpha} & \longmapsto A_{\alpha}' = U A_{\alpha} U^{-1} - (\partial_{\alpha} U) U^{-1}  \, . 
\end{align*}
It is  well-known that if  $(A_0,...A_n)$ satisfies (\ref{curv}),(\ref{1}) so does $(A_0',...,A_n')$.

Hence we may impose a gauge condition. We exclusively study the Lorenz gauge $\partial^{\alpha}A_{\alpha} =0$. Other convenient gauges are the Coulomb gauge $\partial^j A_j=0$ and the temporal gauge $A_0 =0$. It is well-known that for the low regularity well-posedness problem for the Yang-Mills equation a null structure for some of the nonlinear terms plays a crucial role.This was first detected by Klainerman and Machedon \cite{KM}, who proved global well-posedness in the case of three space dimensions in temporal and in Coulomb gauge in energy space. The corresponding result in Lorenz gauge, where the Yang-Mills equation can be formulated as a system of nonlinear wave equations, was shown by Selberg and Tesfahun \cite{ST}, who discovered 
that also in this case some of the nonlinearities have a null structure. This allows to rely on some of the methods that were previously used for the Maxwell-Dirac equation in \cite{AFS1} and the Maxwell-Klein-Gordon equation in \cite{ST1}. Tesfahun \cite{Te} improved the local well-posedness result to data without finite energy, namely for $(A(0),(\partial_t A)(0) \in H^s \times H^{s-1}$ and $(F(0),(\partial_t F)(0) \in H^r \times H^{r-1}$ with $s > \frac{6}{7}$ and $r > -\frac{1}{14}$, by discovering an additional partial null structure. Local well-posedness in energy space was also given by Oh \cite{O} using a new gauge, namely the Yang-Mills heat flow. He was also able to shows that this solution can be globally extended \cite{O1}. Tao \cite{T1} showed local well-posedness for small data in $H^s \times H^{s-1}$ for $ s > \frac{3}{4}$ in temporal gauge. Tao's result was generalized to space dimensions $n \ge 3$ by the author \cite{P}. In space dimension $n$ the critical regularity with respect to scaling is $s= \frac{n}{2}-1$ . In the case $n=4$ where the energy space is critical. Klainerman and Tataru \cite{KT} proved  small data local well-posedness for a closely related model problem in Coulomb gauge for $s>1$. Klainerman and Selberg \cite{KS} treated the local well-posedness problem with minimal regularity for some systems of nonlinear wave equations. Especially, they showed local well-posedness for a model problem related to the Yang-Mills system in the almost critical region, where $s > \frac{n}{2} -1$. Recently the result \cite{KT} was significantly improved by Krieger and Tataru \cite{KrT}, who were able to show global well-posedness for data with small energy. Sterbenz \cite{St} considered also the four-dimensional case in Lorenz gauge and proved global well-posedness for small data in Besov space $\dot{B}^{1,1} \times \dot(B)^{0,1}$. In high space dimension $n \ge 6$ (and $n$ even)  Krieger and Sterbenz \cite{KrSt} proved global well-posedness for small data in the critical Sobolev space.

In the present paper we consider the low regularity local-wellposedness problem for large data for the Yang-Mills system in Lorenz gauge and space dimension $n \ge 3$ . In the case $n\ge 4$ our main result is local well-posedness for $s > \frac{n}{2}-\frac{7}{8}$ and $r > \frac{n}{2} -\frac{7}{4}$ , in the case $n=3$ we obtain local well-posedness for $s > \frac{3}{4}$ and $r >-\frac{1}{8}$, where existence holds in $ A \in C^0([0,T],H^s) \cap C^1([0,T],H^{s-1}) \, , \, F \in C^0([0,T],H^r) \cap C^1([0,T],H^{r-1}) $ and (existence and) uniqueness in a certain subspace (Theorem \ref{Theorem1.1} and Corollary \ref{Cor}). In the case $n=3$ this is an improvement of Tesfahun's  result \cite{Te}, whereas in the case $n \ge 4$ the results for the full Yang-Mills system in Lorenz gauge and large data are also new.
 Crucial for this result are on one hand the methods developed in the papers by Selberg-Tesfahun \cite{ST} and Tesfahun \cite{Te}, especially their detection of the null structure in most - unfortunately not all - the  critical nonlinear terms. On the other hand we have to consider in space dimensions $n \ge 4$ a more sophisticated solution space, where we rely on the methods by Klainerman and Selberg \cite{KS} for a model problem for Yang-Mills, which ignores the gauge condition. We modify their solution space appropriately and show that its main properties are preserved. We were unable to come down to the critical value $s= \frac{n}{2}-1$ , which is prevented mainly by one of the nonlinear terms, for which no null structure is known and which leads to the estimate (\ref{32}). \cite{ST} and \cite{Te} used solution spaces of wave-Sobolev type $H^{s,b}$ , which are closely related to the Bourgain-Klainerman-Machedon spaces $X^{s,b}$ , for which a convenient  atlas of bilinear estimates was proven in \cite{AFS} and \cite{AFS1} in dimension $n \le 3$ and stated in arbitrary dimension. We give a proof for a special case in $n \ge 4$ and also rely on a paper by Lee and Vargas \cite{LV}, who obtain $L^p_t L^q_x$ - estimates for products of solutions of the wave equation.
If one uses solution spaces of $H^{s,b}$ - type it seems to be impossible to obtain our results  in space dimensions $n \ge 4$, because some of the bilinear estimates which we need simply fail. Therefore it is necessary to modify the solution spaces appropriately.

In chapter 2 we recall the reformulation of the Yang-Mills equation as a system of nonlinear wave equations and state our main theorem (Theorem \ref{Theorem1.1} and Corollary \ref{Cor}). We also fix some notation.
Chapter 3 contains the bilinear estimates in wave-Sobolev spaces. Moreover we define the solution spaces and state its fundamental properties. We reduce the local well-posedness problem to a suitable set of nonlinearities in Proposition \ref{Prop1.6}, where we complete rely on \cite{KS}. In chapter 4 we formulate the Yang-Mills equations in final form - using the whole null structure - and the necessary nonlinear estimates as in \cite{Te}. We also review some well-known properties of the standard null forms and the additional one detected in \cite{Te}.
In chapter 4 and 5 we prove the estimates for the nonlinearities for $n \ge 4$ and $n=3$ , respectively.

{\bf Acknowledgment:} I thank Axel Gr\"unrock who pointed out the paper by Lee-Vargas \cite{LV} to me.

\section{Main results}
Expanding (\ref{1}) in terms of the gauge potentials $\left\{A_\alpha\right\}$, we obtain: 
 \begin{equation}\label{YM2}
  \square A_\beta = \partial_\beta\partial^\alpha A_\alpha- [\partial^\alpha A_\alpha, A_\beta] - [A^\alpha,\partial^\alpha A_\beta] - 
  [A^\alpha, F_{\alpha\beta}].
\end{equation}
If we now impose the Lorenz gauge condition,
the system \eqref{YM2} reduces to the nonlinear wave equation
 \begin{equation}\label{YM3}
  \square A_\beta = - [A^\alpha,\partial_\alpha A_\beta] -   [A^\alpha, F_{\alpha\beta}].
\end{equation}
In addition, regardless of the choice of gauge, $F$ satisfies the wave equation 
    \begin{equation}\label{YMF1}
 \begin{split}
   \square F_{\beta\gamma}&=-[A^\alpha,\partial_\alpha F_{\beta\gamma}] - \partial^\alpha[A_\alpha,F_{\beta\gamma}] - \left[A^\alpha,[A_\alpha,F_{\beta\gamma}]\right]
   \\
   & \quad - 2[F^{\alpha}_{{\;\;\,}\beta},F_{\gamma\alpha}].
   \end{split}
 \end{equation}
  Indeed,  this will follow if we apply $D^\alpha$ to the 
 Bianchi identity
$$
  D_\alpha F_{\beta\gamma} + D_\beta F_{\gamma\alpha} + D_\gamma F_{\alpha\beta} = 0
$$
and simplify the resulting expression using the commutation identity
$$
  D_\alpha D_\beta X - D_\beta D_\alpha X = [F_{\alpha\beta},X]
$$
and  (\ref{1})  (\cite{ST}).

Expanding the second and fourth terms in \eqref{YMF1}, and also imposing the Lorenz gauge, yields 
\begin{equation}\label{YMF2}
\begin{split}
       \square F_{\beta\gamma}&= - 2[A^\alpha,\partial_\alpha F_{\beta\gamma}]
      + 2[\partial_\gamma A^\alpha, \partial_\alpha A_\beta]
      - 2[\partial_\beta A^\alpha, \partial_\alpha A_\gamma]
      \\
      &\quad + 2[\partial^\alpha A_\beta , \partial_\alpha A_\gamma]
                 + 2[\partial_\beta A^\alpha, \partial_\gamma A_\alpha] - [A^\alpha,[A_\alpha,F_{\beta\gamma}]] 
                 \\
                 &\quad  +2[F_{\alpha\beta},[A^\alpha,A_\gamma]]- 2[F_{\alpha\gamma},[A^\alpha,A_\beta]]
                      - 2[[A^\alpha,A_\beta],[A_\alpha,A_\gamma]] .  
                         \end{split}    
\end{equation}

 Note on the other hand by expanding the last term in the right hand side of \eqref{YM3}, we obtain 
 \begin{equation}\label{YM4}
  \square A_\beta = - 2[A^\alpha,\partial_\alpha A_\beta] + [A^\alpha,\partial_\beta A_\alpha] - 
  [A^\alpha, [A_\alpha,A_\beta]].
\end{equation}

We want to solve the system \eqref{YMF2}-\eqref{YM4} simultaneously for $A$ and $F$.
So to pose the Cauchy problem for this system, we consider initial data for $(A,F)$ at $t=0$ :
\begin{equation}\label{Data-AF}
A(0) = a, \quad \partial_t A(0) = \dot a,
        \quad
    F(0) =f, \quad \partial_t F(0) = \dot f.
 \end{equation}

In fact, the initial data for $F$ can be determined from $(a, \dot a)$ as follows:
\begin{equation}\label{f}
\left\{
\begin{aligned}
  f_{ij} &= \partial_i a_j - \partial_j a_i + [a_i,a_j],
  \\
  f_{0i} &= \dot a_i - \partial_i a_0 + [a_0,a_i],
\\
  \dot f_{ij} &= \partial_i \dot a_j - \partial_j \dot a_i + [\dot a_i,a_j]+[ a_i, \dot a_j],
  \\
 \dot f_{0i} &= \partial^j f_{ji} +[a^\alpha, f_{\alpha i}] 
\end{aligned}
\right.
\end{equation}
where the first three expressions come from \eqref{curv} whereas 
the last one comes from (\ref{1}) with $\beta=i$.

Note that the Lorenz gauge condition $\partial^\alpha A_\alpha=0$ and (\ref{1}) with $\beta=0$ impose the constraints 
\begin{equation}\label{Const}
\dot a_0= \partial^i a_i,
\quad
   \partial^i f_{i0} + [a^i, f_{i0}]=0.
\end{equation}

Now we formulate our main theorem.
\begin{theorem}
\label{Theorem1.1}
If $n \ge 4$ , assume that $s$ and $r$ satisfy the following conditions:
\begin{align*}
s > \frac{n}{2}-\frac{7}{8} \, ,& \quad \quad r >\frac{n}{2}-\frac{7}{4} \, , \quad \quad \quad r < s < r+1 \, .\\
3r-2s > \frac{n}{2} - \frac{7}{2}\, , &  \quad 2r-s > - \frac{3}{4} \, \, \,(\text{if} \,\,\, n=4) \, ,\\
2s-r > \frac{n}{2} \, , & \quad 3s-2r > \frac{n}{2} + \half \, .
\end{align*}
If $n=3$ , assume :
\begin{align*}
s > \frac{3}{4} \, , & \quad r >-\frac{1}{8} \, , \quad  r < s < r+1 \, , \\
2r-s > -1\, , &  \quad 2s-r >  \frac{3}{2} \, ,  \quad 3s-2r > \frac{7}{4} 
\end{align*}
Given initial data $(a,\dot{a}) \in H^s \times H^{s-1}$ ,  $(f,\dot{f}) \in H^r \times H^{r-1}$ , there exists a time $T > 0$ , $T=T(\|a\|_{H^s},\|\dot{a}\|_{H^{s-1}} , \|f\|_{H^r} , \|\dot{f}\|_{H^{r-1}})$ , such that the Cauchy problem (\ref{YMF2}),(\ref{YM4}),(\ref{Data-AF}) has a unique solution $A \in F^s_T$ , $F \in G^r_T$ (these spaces are defined in Def. \ref{Def.1.2}). This solution has the regularity
$$ A \in C^0([0,T],H^s) \cap C^1([0,T],H^{s-1}) \quad , \quad F \in C^0([0,T],H^r) \cap C^1([0,T],H^{r-1}) \, . $$
\end{theorem}
{\bf Remark:} {\bf Case $n \ge 4$ .} 1. The assumptions on $s$ and $r$ imply $6s-\frac{3}{2}n > 3r > 2s + \frac{n}{2} -\frac{7}{2}$,  which can only be fulfilled, if $s > \frac{n}{2}-\frac{7}{8}$ , and therefore $r > \frac{n}{2} - \frac{7}{4}$.  One easily checks that the choice $s= \frac{n}{2}-\frac{7}{8} + \epsilon$ , $ r = \frac{n}{2} - \frac{7}{4}+\epsilon$ satisfies our assumptions, if $\epsilon > 0$ is small enough. \\
2. The following estimate is automatically fulfilled $ 2r-2 > \frac{n}{2}- 3 $ , because $2r-s=3r-2s -r+s > \frac{n}{2} - \frac{7}{2}$ .  \\
{\bf Case $n=3$ .}
 1. The assumptions on $s$ and $r$ imply $s > \frac{3}{4}$ and $r > \frac{s}{2} - \half$ , which implies $r >  - \frac{1}{8}$ . One easily checks that the choice $s= \frac{3}{4} + \epsilon$ , $ r = - \frac{1}{8}+\epsilon$ satisfies our assumptions, if $\epsilon > 0$ is small enough. \\
2. The following conditions are automatically fulfilled 
$$ 3r-2s > -2  \,,  \quad 4r-s >-2 \, , \quad 4s-r > 3 \, ,$$
because $3r-2s = (2r-s)+(r-s) > -1-1 = -2$ and $4r-s>(2r-s)+2r > -1+2r > -1-\frac{1}{4} > -2$ .

\begin{Cor}
\label{Cor}
Let $s,r$ fulfill the assumptions of Theorem \ref{Theorem1.1}. Moreover assume that the initial data fulfill (\ref{f}) and (\ref{Const}). Given any $(a,\dot{a}) \in H^{r+1} \times H^r$ , there exists a time $T > 0$ , $T=T(\|a\|_{H^s},\|\dot{a}\|_{H^{s-1}} , \|f\|_{H^r} , \|\dot{f}\|_{H^{r-1}})$ , such that the solution $(A,F)$ of Theorem \ref{Theorem1.1} satisfies the Yang-Mills system (\ref{curv}),(\ref{1}) with Cauchy data $(a,\dot{a})$ and the Lorenz gauge condition $\partial^{\alpha} A_{\alpha} =0$ .
\end{Cor}
\begin{proof}[Proof of the Corollary]
The solution $(A,F)$ does not necessarily fulfill the Lorenz gauge condition and (\ref{curv}), i.e. $F=F[A]$ . If however the conditions (\ref{f}) and (\ref{Const}) are assumed then these properties are satisfied and $(A,F)$ is a solution of the Yang-Mills system (\ref{curv}),(\ref{1}) with Cauchy data $(a,\dot{a})$. This was shown in \cite{ST}, Remark 2. 
\end{proof}

{\bf Remarks:} 1. Because $s < r+1$ by assumption the potential $A$ possibly loses some regularity compared to its data, whereas this is not the case for $F$ , which is the decisive factor, whereas the regularity of $A$ is of minor interest. \\
2. If $(a,\dot{a}) \in H^{r+1} \times H^r$ , then $(f,\dot{f})$ , defined by (\ref{f}) , fulfills $(f,\dot{f}) \in H^r \times H^{r-1}$,  as one easily checks.

Let us fix some notation.
We denote the Fourier transform with respect to space and time  by $\,\,\widehat{}\,$ . 
 $\Box = \partial_t^2 - \Delta$ is the d'Alembert operator,
$a\pm := a \pm \epsilon$ for a sufficiently small $\epsilon >0$ , and $\langle \,\cdot\, \rangle := (1+|\cdot|^2)^{\frac{1}{2}}$ .

The standard wave-Sobolev spaces $H^{s,b}$ of Bourgain-Klainerman-Machedon type are the completion of the Schwarz space ${\mathcal S}(\R^{1+n})$ with norm
$$ \|u\|_{H^{s,b}} = \| \langle \xi \rangle^s \langle  |\tau| - |\xi| \rangle^b \widehat{u}(\tau,\xi) \|_{L^2_{\tau \xi}} \, . $$ 
We also define $H^{s,b}_T$ as the space of the restrictions of functions in $H^{s,b}$ to $[0,T] \times \mathbb{R}^n$.

Let $\Lambda^{\alpha}$, $\Lambda_{+}^{\alpha}$ and $\Lambda_{-}^{\alpha}$
be the multipliers with symbols  $$
\langle\xi \rangle^\alpha,
\quad \langle|\tau|  + |\xi|\rangle^\alpha,
\quad \langle|\tau|-|\xi|\rangle^\alpha.
$$
Similarly let $D^{\alpha}$,
$D_{+}^{\alpha}$ and $D_{-}^{\alpha}$ be the multipliers with symbols $$
\abs{\xi}^\alpha,
\quad (\abs{\tau} + \abs{\xi})^\alpha,
\quad ||\tau|-|\xi||^\alpha,
$$
respectively.

Let $\partial$ denote the collection of space and time derivatives.

If $u, v \in {\mathcal S}'$ and $\widehat u, \widehat v$ are tempered
functions, we write $u \preceq v$ iff $\abs{\widehat u} \le \widehat v$,
and $\precsim$ means $\preceq$ up to a constant. If $u = (u^1,\dots,u^N)$
and $v = (v^1,\dots,v^N)$, then $u \preceq v$ (resp. $u \precsim v$) means
$u^I \preceq v^I$ (resp. $u^I \precsim v^I$) for $I = 1, \dots, N$.

\section{Preliminaries}
The Strichartz type estimates for the wave equation are given in the next proposition.
\begin{prop}
\label{Str}
If $n \ge 2$ and
\begin{equation}
2 \le q \le \infty, \quad
2 \le r < \infty, \quad
\frac{2}{q} \le (n-1) \left(\half - \frac{1}{r} \right), 
\end{equation}
then the following estimate holds:
$$ \|u\|_{L_t^q L_x^r} \lesssim \|u\|_{H^{\frac{n}{2}-\frac{n}{r}-\frac{1}{q},\half+}} \, ,$$
especially in the case $n \ge 4$ :
$$ \|u\|_{L^2_t L^r_x} \lesssim \|u\|_{H^{\frac{n-1}{2}-\frac{n}{r},\half +}} $$
for $ \frac{2(n-1)}{n-3} \le r < \infty $ .
\end{prop}
\begin{proof}
This is the Strichartz type estimate, which can be found for e.g. in \cite{GV}, Prop. 2.1, combined with the transfer principle.
\end{proof}

An immediate consequence is the following modified Strichartz estimate.
\begin{prop}
\label{mStr}
If $n \ge 4$ , $ 2 \le r \le \frac{2(n-1)}{n-3}$ , one has the estimate
$$ \|u\|_{L^2_t L^r_x} \lesssim \|u\|_{H^{\frac{n+1}{2}\left(\half-\frac{1}{r}\right),\frac{n-1}{2}\left(\half-\frac{1}{r}\right)+}} \lesssim  \|u\|_{H^{\frac{n+1}{2}\left(\half-\frac{1}{r}\right),\half+}} \, . 
$$
\end{prop}
\begin{proof}
The last estimate is trivial. For the first one we
interpolate the trivial identity $\|u\|_{L^2_t L^2_x} = \|u\|_{H^{0,0}} $ with the estimate
$$ \|u\|_{L^2_t L^{\frac{2(n-1)}{n-3}}_x} \lesssim \|u\|_{H^{\frac{n+1}{2(n-1)},\half+}} \, , $$
which holds by Prop. \ref{Str}.
\end{proof}

The following product estimates for wave-Sobolev spaces were proven in \cite{AFS}.
\begin{prop}
\label{Prop.1.2'} Let $n=3$ .
 For $s_0,s_1,s_2,b_0,b_1,b_2 \in {\mathbb R}$ and $u,v \in   {\mathcal S} ({\mathbb R}^{3+1})$ the estimate
$$\|uv\|_{H^{-s_0,-b_0}} \lesssim \|u\|_{H^{s_1,b_1}} \|v\|_{H^{s_2,b_2}} $$ 
holds, provided the following conditions are satisfied:
\begin{align*}
\nonumber
& b_0 + b_1 + b_2 > \frac{1}{2} \, ,
& b_0 + b_1 \ge 0 \, ,\quad \qquad  
& b_0 + b_2 \ge 0 \, ,
& b_1 + b_2 \ge 0
\end{align*}
\begin{align*}
\nonumber
&s_0+s_1+s_2 > 2 -(b_0+b_1+b_2) \\
\nonumber
&s_0+s_1+s_2 > \frac{3}{2} -\min(b_0+b_1,b_0+b_2,b_1+b_2) \\
\nonumber
&s_0+s_1+s_2 > 1 - \min(b_0,b_1,b_2) \\
\nonumber
&s_0+s_1+s_2 > 1 \\
 &(s_0 + b_0) +2s_1 + 2s_2 > \frac{3}{2} \\
\nonumber
&2s_0+(s_1+b_1)+2s_2 > \frac{3}{2} \\
\nonumber
&2s_0+2s_1+(s_2+b_2) > \frac{3}{2}
\end{align*}
\begin{align*}
\nonumber
&s_1 + s_2 \ge \max(0,-b_0) \, ,\quad
\nonumber
s_0 + s_2 \ge \max(0,-b_1) \, ,\quad
\nonumber
s_0 + s_1 \ge \max(0,-b_2)   \, .
\end{align*}
\end{prop}

\vspace{1em}

The following proposition was proven by \cite{KT}.
\begin{prop}
\label{TheoremE}
Let $n \ge 2$, and let $(q,r)$ satisfy:
$$
2 \le q \le \infty, \quad
2 \le r < \infty, \quad
\frac{2}{q} \le (n-1) \left(\half - \frac{1}{r} \right)
$$
Assume that
\begin{gather*}
0 < \sigma < n - \frac{2n}{r} - \frac{4}{q}, \\
s_1, s_2 < \frac{n}{2} - \frac{n}{r} - \frac{1}{q}, \\
s_1 + s_2 + \sigma = n - \frac{2n}{r} - \frac{2}{q}. \end{gather*} 
then
$$
\|D^{-\sigma}(uv)\|_{L_t^{q/2} L_x^{r/2}}
\lesssim
\|u\|_{H^{s_1,\half+}}
\|v\|_{H^{s_2,\half+}} \, .
$$
\end{prop}

The following product estimate for wave-Sobolev spaces is a special case of the very convenient much more general atlas formulated by \cite{AFS} in arbitrary dimension, but proven only in the case $1\le n \le 3$. (\cite{AFS} and \cite{AFS1}). Therefore we have to give a proof.

\begin{prop}
\label{Prop.1.2}
Assume $n \ge 4$ and
$$s_0+s_1+s_2 > \frac{n-1}{2} \, , \, (s_0+s_1+s_2)+s_1+s_2 > \frac{n}{2} \, , \, s_0+s_1 \ge 0 \, , \, s_0+s_2 \ge 0 \, , \, s_1+s_2 \ge 0 \, . $$
The following estimate holds:
$$ \|uv\|_{H^{-s_0,0}} \lesssim \|u\|_{H^{s_1,\half+}} \|v\|_{H^{s_2,\half+}} \, . $$
\end{prop}
\begin{proof}
We have to prove
\begin{align*}
I & := \int_* \frac{\widehat{u}_1(\xi_1,\tau_1)}{\angles{\xi_1}^{s_1} \angles{|\xi_1|-|\tau_1|}^{\half+}}  \frac{\widehat{u}_2(\xi_2,\tau_2)}{\angles{\xi_2}^{s_2} \angles{|\xi_2|-|\tau_2|}^{\half+}}  \frac{\widehat{u}_0(\xi_0,\tau_0)}{\angles{\xi_0}^{s_0}} \lesssim \|u_1\|_{L^2_{xt}}
\|u_2\|_{L^2_{xt}} \, .
\end{align*}
Here * denotes integration over $\xi_0+\xi_1+\xi_2=0 $ and $\tau_0+\tau_1+\tau_2=0$ . Remark, that we may assume that the Fourier transforms are nonnegative. We consider different regions. \\
1. If $|\xi_0| \sim |\xi_1| \gtrsim |\xi_2|$ and $s_2 \ge 0$, we obtain
\begin{align*}
I \sim \int_* \frac{\widehat{u}_1(\xi_1,\tau_1)}{\angles{\xi_1}^{s_1+s_0} \angles{|\xi_1|-|\tau_1|}^{\half+}}  \frac{\widehat{u}_2(\xi_2,\tau_2)}{\angles{\xi_2}^{s_2} \angles{|\xi_2|-|\tau_2|}^{\half+}}  \widehat{u}_0(\xi_0,\tau_0) \, .
\end{align*}
Thus we have to show
$$ \|uv\|_{L^2_{xt}} \lesssim \|u\|_{H^{s_1+s_0,\half+}} \|v\|_{H^{s_1,\half+}} \, . $$
By Prop. \ref{Str} we obtain
$$ \|uv\|_{L^2_{xt}} \lesssim \|u\|_{L^{\infty}_t L^2_x} \|v\|_{L^2_t L^{\infty}_x} \lesssim \|u\|_{H^{0,\half+}} \|v\|_{H^{\frac{n-1}{2}+,\half+}} $$
and also
$$ \|uv\|_{L^2_{xt}} \lesssim  \|u\|_{H^{\frac{n-1}{2}+,\half+}}      \|v\|_{H^{0,\half+}} \, .$$
Bilinear interpolation gives for $0 \le \theta \le 1 $ :
$$ \|uv\|_{L^2_{xt}} \lesssim \|u\|_{H^{\frac{n-1}{2}(1-\theta)+,\half+}} \|v\|_{H^{\frac{n-1}{2}\theta+,\half+}}  \, , $$
so that 
$$   \|uv\|_{L^2_{xt}} \lesssim \|u\|_{H^{s_1+s_0,\half+}} \|v\|_{H^{s_2,\half+}} \, , $$
if $ s_0+s_1+s_2 > \frac{n-1}{2} $ and $s_1+s_0 \ge 0$ .\\
2. If $|\xi_0| \sim |\xi_2| \gtrsim |\xi_1|$ and $s_1 \ge 0$ , we obtain similarly
$$  \|uv\|_{L^2_{xt}} \lesssim \|u\|_{H^{s_2+s_0,\half+}} \|v\|_{H^{s_1,\half+}} \, , $$
if $ s_0+s_1+s_2 > \frac{n-1}{2} $ and $s_2+s_0 \ge 0$ .\\
3. If $|\xi_1| \ge |\xi_2|$ , $s_0 \le 0$ and $s_2 \ge 0$, we have $|\xi_0| \lesssim |\xi_1| $ , so that
$ \angles{\xi_0}^{-s_0} \lesssim \angles{\xi_1}^{-s_0}$ and we obviously obtain the same result as in 1. \\
4. If $|\xi_1| \le |\xi_2|$ , $s_0 \le 0$ and $s_1 \ge 0$ , we obtain the same result as in 2. \\
5.  If $|\xi_0| \sim |\xi_1| \gtrsim |\xi_2|$ and $s_2 \le 0$ we obtain
$$ I \lesssim \int_* \frac{\widehat{u}_1(\xi_1,\tau_1)}{\angles{\xi_1}^{s_0+s_1+s_2} \angles{|\xi_1|-|\tau_1|}^{\half+}}  \frac{\widehat{u}_2(\xi_2,\tau_2)}{\angles{|\xi_2|-|\tau_2|}^{\half+}}  \widehat{u}_0(\xi_0,\tau_0)  \lesssim \|u_1\|_{L^2_{xt}} \|u_2\|_{L^2_{xt}} \, , $$
because under our asumption $s_0+s_1+s_2 > \frac{n-1}{2}$ we obtain by Prop. \ref{Str}:
$$ \|uv\|_{L^2_{xt}} \le \|u\|_{L^2_t L^{\infty}_x} \|v\|_{L^{\infty}_x L^2_t} \lesssim \| u \|_{H^{s_0+s_1+s_2,\half+}} \|v\|_{H^{0,\half+}} \, . $$
6. If $|\xi_1| \ge |\xi_2|$ , $s_2 \le 0$ and $s_0 \le 0$ , or \\
7. If $|\xi_0| \sim |\xi_2| \ge |\xi_1|$ and $s_1 \le 0$ , or \\
8. If $|\xi_1| \le |\xi_2|$ , $s_1 \le 0$ and $s_0 \le 0$ , the same argument applies.

Thus we are done, if $s_0 \le 0$ , and also, if $s_0 \ge 0$ , and $|\xi_0| \sim |\xi_2 | \ge |\xi_1|$ or $ |\xi_0| \sim |\xi_1 | \ge |\xi_2|$ . 

It remains to consider the following case: $|\xi_0| \ll |\xi_1| \sim |\xi_2|$ and $s_0 > 0$ . We apply Prop. \ref{TheoremE} which gives
$$ \|uv\|_{H^{-s_0,0}} \lesssim \|u\|_{H^{s_1,\half+}} \|v\|_{H^{s_2,\half+}} \, , $$
under the conditions $0 < s_0 < \frac{n}{2} - 1 $ , $s_0+s_1+s_2 = \frac{n-1}{2} $ and $s_1,s_2 < \frac{n-1}{4}$ . The last condition is not necessary in our case $|\xi_1| \sim |\xi_2|$ . Remark that this implies $s_1+s_2 > \half$ , so that $s_0+s_1+s_2+s_1+s_2 > \frac{n}{2}$ .  , The second condition can now be replaced by $s_0+s_1+s_2 \ge \frac{n-1}{2}$ , because we consider inhomogeneous spaces.

Finally we consider the case $|\xi_0| \ll |\xi_1| \sim |\xi_2|$ and $s_0 \ge \frac{n}{2}-1 $ .
If $s_0 > \frac{n}{2} $ and $s_1+s_2 \ge 0$ we obtain the claimed estimate by Sobolev
$$ \|uv\|_{H^{-s_0,0}} \lesssim \|u\|_{H^{0,\half+}} \|v\|_{H^{0,\half+}} \le  \|u\|_{H^{s_1,\half+}} \|v\|_{H^{s_2,\half+}} \, . $$
We now interpolate the special case
$$ \|uv\|_{H^{-\frac{n}{2}-,0}} \lesssim \|u\|_{H^{0,\half+}} \|v\|_{H^{0,\half+}}  $$
with the following estimate
$$ \|uv\|_{H^{1-\frac{n}{2}+,0}} \lesssim \|u\|_{H^{\frac{1}{4}+,\half+}} \|v\|_{H^{\frac{1}{4}+,\half+}} \, , $$
which follows from Prop. \ref{TheoremE} . We obtain
$$\|uv\|_{H^{-s_0-,0}} \lesssim \|u\|_{H^{k+,\half+}} \|v\|_{H^{k+,\half+}} \, , $$
where $s_0 = (1-\theta)\frac{n}{2} - \theta(1-\frac{n}{2}) = \frac{n}{2} - \theta \, \Leftrightarrow \, \theta = \frac{n}{2}-s_0$ , $0 \le \theta \le 1 $ , $k=\frac{\theta}{4} = \frac{n}{8}-\frac{s_0}{4}.$  Using our asumption  $(s_0+s_1+s_2)+s_1+s_2 > \frac{n}{2} \, \Leftrightarrow \, \frac{n}{2}-s_0 < 2(s_1+s_2),$  we obtain $0 \le k < \frac{s_1+s_2}{2}$ .  Because $|\xi_1| \sim |\xi_2|$ , we obtain
$$ \|uv\|_{H^{-s_0,0}} \lesssim \|u\|_{H^{s_1,\half+}} \|v\|_{H^{s_2,\half+}} $$
for $(s_0+s_1+s_2)+s_1+s_2 > \frac{n}{2}$ and $s_1+s_2 \ge 0$ .
\end{proof}

\begin{Cor}
\label{Cor.1.1}
Under the assumptions of Prop. \ref{Prop.1.2}
$$ \|uv\|_{H^{-s_0,0}} \lesssim \|u\|_{H^{s_1,\half-}} \|v\|_{H^{s_2,\half-}}  \, . $$
\end{Cor}
\begin{proof}
This follows by bilinear interpolation of the estimate of Prop. \ref{Prop.1.2} with the estimate
$$\|uv\|_{H^{N,0}} \lesssim \|u\|_{H^{N,\frac{1}{4}+}} \|v\|_{H^{N,\frac{1}{4}+}}  \, , $$
where, say, $N > \frac{n}{2}$ , which follows by Sobolev apart from the special case $s_1=-s_2,$  in which we interpolate with the estimate
$$\|uv\|_{H^{-N,0}} \lesssim \|u\|_{H^{N,\frac{1}{4}+}} \|v\|_{H^{-N,\frac{1}{4}+}}   $$
in order to save the condition $s_1=-s_2$ .
\end{proof}

\begin{Cor}
\label{Cor.1.2}
If $s_1 >\frac{n-1}{2}$ , $2s_1+s_0 > \frac{n}{2}$ , $ s_1+s_0 \ge 0$ , $0 \le \epsilon < \half$
the following estimate holds
$$ \|uv\|_{H^{-s_0,\epsilon}} \lesssim \|u\|_{H^{s_1+O(\epsilon),\half+}} \|v\|_{H^{-s_0,\half+}} \, . $$
\end{Cor}
\begin{proof}
By Prop. \ref{Prop.1.2} we have
$$  \|uv\|_{H^{-s_0,0}} \lesssim \|u\|_{H^{s_1,\half+}} \|v\|_{H^{-s_0,\half+}} \, . $$
Moreover for $N > \frac{n}{2}$ we obtain by Sobolev
$$ \|uv\|_{H^{-s_0,\half}} \lesssim \|u\|_{H^{N,\half+}} \|v\|_{H^{-s_0,\half+}}  \, .$$
The result follows by interpolation.
\end{proof}

The following multiplication law is well-known:
\begin{prop} {\bf (Sobolev multiplication law)}
\label{SML}
Let $n\ge 2$ , $s_0,s_1,s_2 \in \R$ . Assume
$s_0+s_1+s_2 > \frac{n}{2}$ , $s_0+s_1 \ge 0$ ,  $s_0+s_2 \ge 0$ , $s_1+s_2 \ge 0$. Then the following product estimate holds:
$$ \|uv\|_{H^{-s_0}} \lesssim \|u\|_{H^{s_1}} \|v\|_{H^{s_2}} \, .$$
\end{prop}

We also need the following bilinear estimates in $H^{s,b}$-spaces, which follow as a special case from the results stated in \cite{AFS} and \cite{AFS1}, but proven in these papers only for $n=2$ and $n=3$. We postpone the proof to the appendix.
\begin{prop}
\label{Prop.3.8}
Let $n \ge 2$ .
Assume $s_0+s_1+s_2 \ge \frac{n}{2} + \epsilon$ , $s_1+s_2 \ge \half$ , $s_0+s_1 \ge 0$ , $s_0+s_2 \ge 0$ . Then the following estimates hold for $\epsilon > 0$ sufficiently small:
\begin{align}
\label{a}
\|uv\|_{H^{-s_0,\half + \epsilon}} & \lesssim \|u\|_{H^{s_1,\half+\epsilon}} \|v\|_{H^{s_2,\half+\epsilon}} \, , \\
\label{b}
\|uv\|_{H^{-s_0,\half -2 \epsilon}} & \lesssim \|u\|_{H^{s_1,\half-2\epsilon}} \|v\|_{H^{s_2,\half+\epsilon}}
\end{align}
\end{prop}

Next we formulate a special case of the fundamental estimates for the $L^q_t L^p_x$-norm of the product of solutions of the wave equation due to Lee-Vargas \cite{LV}.
\begin{prop}
\label{LV}
Assume $n \ge 4$ and
$$ \Box \,u = \Box \,v = 0 $$
in $\R^n \times \R$ .
The estimate
$$ \|uv\|_{L^q_t L^2_x} \lesssim (\|u(0)\|_{\dot{H}^{\alpha_1}} + \|(\partial_t u)(0)\|_{\dot{H}^{\alpha_1-1}})(\|v(0)\|_{\dot{H}^{\alpha_2}} + \|(\partial_t v)(0)\|_{\dot{H}^{\alpha_2-1}}) $$
holds, provided $1< q \le 2$ and
\begin{align}
\label{LV2}
\alpha_1 + \alpha_2 &= \frac{n}{2}-\frac{1}{q} \, , \\
\label{LV4}
\frac{1}{q} & <  \frac{n-1}{4} \, , \\
\label{LV12}
\alpha_1,\alpha_2 & <  \frac{n}{2} + \half - \frac{2}{q} \, .
\end{align}
\end{prop}
\begin{proof}
This follows by easy calculations from \cite{LV}, Theorem 1.1.
\end{proof}
\begin{Cor}
\label{Cor.3.3}
If (\ref{LV2}),(\ref{LV4}) and (\ref{LV12}) are satisfied for some $1<q\le 2$ and additonally $\alpha_1,\alpha_2 \ge 0$, the following estimate holds
$$ \|uv\|_{H^{0,\half-\frac{1}{q}}} \lesssim \|u\|_{H^{\alpha_1,\half+}} \|v\|_{H^{\alpha_2,\half+}} \, . $$
\end{Cor}
\begin{proof}
The proposition implies
$$\| e^{\pm it D} f e^{\pm_1it D} g \|_{L^q_t L^2_x} \lesssim \|f\|_{\dot{H}^{\alpha_1}} \|g\|_{\dot{H}^{\alpha_2}} \, ,$$
so that the claimed estimate follows by the transfer principle combined with the estimate
$\|uv\|_{H^{0,\half-\frac{1}{q}}} \lesssim\|uv\|_{L^q_t L^2_x}$ .
\end{proof}

We now come to the definition of the solution spaces, which are very similar to the spaces introduced by \cite{KS}. We prepare this by defining a modification of the standard $L^q_t L^r_x$-spaces. 
\begin{Def}
\label{Def.1.2}
If $1 \le q, r \le
\infty$, $u \in {\mathcal S}'$ and $\widehat u$ is a tempered function, set $$
\|u\|_{{\mathcal L}_t^q {\mathcal L}_x^r} = \sup \left\{ \int_{\R^{1+n}} \abs{ \widehat u (\tau,\xi)
} \widehat v(\tau,\xi) \, d\tau d\xi : v \in {\mathcal S} , \widehat v \ge 0 ,
 \|v\|_{L^{q'}_t L^{r'}_x} = 1 \right\}, $$
where $1 = \frac{1}{q} + \frac{1}{q'}$ and $1 = \frac{1}{r}
+ \frac{1}{r'}$. Let ${\mathcal L}_t^q {\mathcal L}_x^r$ be the corresponding subspace of
${\mathcal S}'$.
\end{Def}

This is a translation invariant norm and it only
depends on the size of the Fourier transform.  Observe that 
${\mathcal L}^2_t {\mathcal L}^2_x = L^2_t L^2_x$  and
$$
 \|u\|_{{\mathcal L}_t^q {\mathcal L}_x^r} \le \|u\|_{L^{q}_t L^{r}_x}\quad \text{whenever} \quad
\widehat u \ge 0. 
$$

\begin{Def}
\label{Def}
Our solution spaces are defined as follows: \\
$\{ (u,v) \in F^s  \times G^r \} \, , $ where \\
1. in the case $n \ge 4$ :
\begin{align*}
\|u\|_{F^s} &:= \|\Lambda_+ u\|_{H^{s-1,\half+\epsilon}} +\| \Lambda_+^{\half+2\epsilon} \Lambda^{-\frac{3}{4} +3\epsilon} \Lambda_-^{\half} u\|_{{\mathcal L}_t^1 {\mathcal L}_x^{4n}} \\
\|v\|_{G^r} &:= \|\Lambda_+ v\|_{H^{r-1,\half+\epsilon}}  \ ,
\end{align*}
2. in the case $n=3$ :
\begin{align*}
\|u\|_{F^s} &:= \|\Lambda_+ u\|_{H^{s-1,\frac{3}{4}+\epsilon}} \\
\|v\|_{G^r} &:= \|\Lambda_+ v\|_{H^{r-1,\half+\epsilon}}  \ ,
\end{align*}
where $\epsilon > 0$ is sufficiently small. $F^s_T$ and $G^r_T$ denotes the restriction to the time interval $[0,T]$.
\end{Def}
This is a Banach space (\cite{KS}, Prop. 4.2).

Next we recall some fundamental properties of the ${\mathcal L}_t^q {\mathcal L}_x^r$-spaces, which were given by \cite{KS}, starting with a H\"older-type estimate.

\begin{prop}\label{Prop.4.3} Suppose $\frac{1}{q} =
\frac{1}{q_{1}} + \frac{1}{q_{2}}$ and $\frac{1}{r} = \frac{1}{r_{1}} +
\frac{1}{r_{2}}$, where the $q$'s and $r$'s all belong to $[1,\infty]$.
Then $$
\|uv\|_{{\mathcal L}^q_t {\mathcal L}^r_x}  \le \|u\|_{{\mathcal L}^{q_1}_t {\mathcal L}^{r_1}_x}  \| v\|_{L^{q_2}_t  L^{r_2}_x} \, .
$$
for all $v$  with $\widehat
v \ge 0$.
\end{prop}

The following duality argument holds. 
\begin{prop}\label{Prop.4.5} Let $1 \le a,b,q,r \le
\infty$.
{
\renewcommand{\theenumi}{\alph{enumi}}
\renewcommand{\labelenumi}{(\theenumi)}
\begin{enumerate}
\item If
\begin{equation}\label{DualityA}
\|G\|_{L^{a'}_t L^{b'}_x} \lesssim \| \Lambda^{\alpha}
\Lambda_{-}^{\beta} G\|_{L^{q'}_t L^{r'}_x} 
\end{equation} 
for all $G$ ,
 then
$$
\|F\|_{L^{q}_t L^{r}_x} \lesssim \| \Lambda^{\alpha}
\Lambda_{-}^{\beta} F\|_{L^{a}_t L^{b}_x} 
$$
for all $F$ .
\item If \eqref{DualityA} holds for all $G$ with $\widehat G
\ge 0$, then
$$
\|F\|_{{\mathcal L}^{q}_t {\mathcal L}^{r}_x} \lesssim \| \Lambda^{\alpha}
\Lambda_{-}^{\beta} F\|_{{\mathcal L}^{a}_t {\mathcal L}^{b}_x} \, .
$$
for all $F$.
\end{enumerate}
}
\end{prop}
\begin{proof}
\cite{KS}, Proposition 4.5.
\end{proof}

The next proposition shows that a Sobolev type embedding also carries over to the ${\mathcal L}_t^q {\mathcal L}_x^r$-spaces.
\begin{prop}
\label{Cor.4.6} Let $1 \le a,b,q,r \le \infty$ , $\alpha,\beta \in {\mathbb R}$ .
If
$$ \|\Lambda^{\alpha} \Lambda_{-}^{\beta}u\|_{L^q_t L^r_x} \lesssim \|u\|_{ L^a_t L^b_x}  $$
for all $u$ with $\widehat{u} \ge 0$ , then
$$\|\Lambda^{\alpha} \Lambda_{-}^{\beta}u\|_{{\mathcal L}^q_t {\mathcal L}^r_x} \lesssim \|u\|_{ {\mathcal L}^a_t {\mathcal L}^b_x } \, . $$
\end{prop}
\begin{proof} \cite{KS}, Cor. 4.6.
\end{proof}

The following result is also fundamental for the proof of our main theorem.
\begin{prop}\label{Prop.4.8} Let $ n \ge 4$ . If $\frac{2(n-1)}{n-3} \le r < \infty $,  $s = \frac{n}{2} -\frac{n}{r}-\half $ and $\theta > \half$ ,
then
$$ \|u\|_{{\mathcal L}^1_t {\mathcal L}^{r}_x}
\lesssim   \|\Lambda^s \Lambda_-^{\theta} u\|_{{\mathcal L}^{1}_t {\mathcal L}^2_x} \, .$$
\end{prop}
\begin{proof} \cite{KS}, Lemma 4.8.
\end{proof}

We also need an elementary estimate which is used as a tool for replacing $H^{s,-\half+}$-norms by $H^{s,-\half-}$-norms.
\begin{lemma}\label{Lemma8.10} Let $\alpha,\beta
\ge 0$. Then
\begin{align*}
\Lambda_{-}^{-\beta}(uv) &\precsim \Lambda_-^{-\alpha-\beta} (\Lambda_+^{\alpha} u \Lambda_+^{\alpha}v) \, , \\
\Lambda_{-}^{-\beta}(uv) &\precsim
\Lambda_{-}^{-\alpha-\beta}(u \Lambda^{\alpha} v) + u \Lambda^{-\beta} v
\end{align*}
for all $u$ and $v$ with $\widehat u, \widehat v \ge 0$. \end{lemma}
\begin{proof}
\cite{KS}, Lemma 8.10.
\end{proof}

Finally, we formulate the fundamental theorem which allows to reduce the local well-posedness for a system of nonlinear wave equations to suitable estimates for the nonlinearities. It is also essentially contained in the paper by \cite{KS}. 

\begin{prop}
\label{Prop1.6} Let $n \ge 3$ .
Let $u_0 \in H^s$ , $u_1 \in H^{s-1}$ , $v_0 \in H^r$ , $v_1 \in H^{r-1}$ be given. Assume that
\begin{align*}
\| \Lambda_+^{-1} \Lambda_-^{\epsilon-1} {\mathcal M}(u,\partial u,v , \partial v)\|_{F^s} & \le \omega_1(\|u\|_{F^s},\|v\|_{G^r}) \, , \\
\| \Lambda_+^{-1} \Lambda_-^{\epsilon-1} {\mathcal N}(u,\partial u,v , \partial v)\|_{G^r} & \le \omega_2(\|u\|_{F^s},\|v\|_{G^r}) \, , 
\end{align*}
and
\begin{align*}
&\| \Lambda_+^{-1} \Lambda_-^{\epsilon-1} ({\mathcal M}(u,\partial u,v , \partial v)\|_{F^s} - {\mathcal M}(u',\partial u',\partial v'))\|_{F^s} \\
&+
\| \Lambda_+^{-1} \Lambda_-^{\epsilon-1} ({\mathcal N}(u,\partial u,v , \partial v)\|_{G^r} - {\mathcal N}(u',\partial u',v' , \partial v')\|_{G^r} \\ &\le \omega(\|u\|_{F^s},\|u'\|_{F^s},\|v\|_{G^r}, \|v'\|_{G^r}) (\|u-u'\|_{F^s} + \|v-v'\|_{G^r}) \, , \\
\end{align*}
where $\omega,\omega_1,\omega_2$ are continuous functions with $\omega(0,0,0,0) = \omega_1(0,0) = \omega_2(0,0) = 0$.
Then the Cauchy problem
$$ \Box \, u = {\mathcal M}(u,\partial u,v,\partial v) \quad , \quad \Box \, v = {\mathcal N}(u,\partial u,v,\partial v) $$
with data
$$u(0) = u_0 \, , \, (\partial_t u)(0) = u_1 \, , \, v(0)= v_0 \, , \,(\partial_t v)(0) = v_1 $$
is locally well-posed, i.e. , there exists $T>0$ , such that there exists a unique solution $u \in F^s_T$ , $v \in G^r_T$ .
\end{prop}
\begin{proof}
This is proved by the contraction mapping principle provided the solution space fulfills suitable assumptions. The case of a single equation $\Box \, u = {\mathcal M}(u,\partial u)$ and the solution space $F^s$ was proven by \cite{KS}, Theorems 5.4 and 5.5, Propositions 5.6 and 5.7. Our case is a straightforward modification of their results, thus we omit the proof.
\end{proof}

\section{Reformulation of the problem and null structure}
The reformulation of the Yang-Mills equations and the reduction of our main theorem to nonlinear estimates is completely taken over from Tesfahun \cite{Te} (cf. also the fundamental paper by Selberg and Tesfahun \cite{ST}).

The standard null forms are given by
\begin{equation}\label{OrdNullforms}
\left\{
\begin{aligned}
Q_{0}(u,v)&=\partial_\alpha u \partial^\alpha v=-\partial_t u \partial_t v+\partial_i u \partial^j v,
\\
Q_{\alpha\beta}(u,v)&=\partial_\alpha u \partial_\beta v-\partial_\beta u \partial_\alpha v.
\end{aligned}
        \right.
\end{equation}
For $\mathfrak g$-valued $u,v$, define a commutator version of null forms by 
\begin{equation}\label{CommutatorNullforms}
\left\{
\begin{aligned}
  Q_0[u,v] &= [\partial_\alpha u, \partial^\alpha v] = Q_0(u,v) - Q_0(v,u),
  \\
  Q_{\alpha\beta}[u,v] &= [\partial_\alpha u, \partial_\beta v] - [\partial_\beta u, \partial_\alpha v] = Q_{\alpha\beta}(u,v) + Q_{\alpha\beta}(v,u).
\end{aligned}
\right.
\end{equation}

 Note the identity
\begin{equation}\label{NullformTrick}
  [\partial_\alpha u, \partial_\beta u]
  = \frac12 \left( [\partial_\alpha u, \partial_\beta u] - [\partial_\beta u, \partial_\alpha u] \right)
  = \frac12 Q_{\alpha\beta}[u,u].
\end{equation}

Define 
\begin{equation}\label{NewNull} 
  \mathcal{Q}[u,v] = - \frac12 Q_{jk}\left[\Lambda^{-1}(R^j u^k - R^k u^j),v\right]
  - Q_{0j}\left[R^j u_0, v \right],
\end{equation}
where 
$R_i = \Lambda^{-1}\partial_i $ is the Riesz transform.

We follow Tesfahun \cite{Te} in the following generalizing his 3-dimensional results to arbitrary dimension $n \ge 3$.

We split the spatial part $\mathbf A=(A_1,..., A_n)$ of the potential into divergence-free and curl-free parts and a smoother part:
\begin{equation}\label{SplitA}  
\mathbf A =  A^{df} + A^{cf} + \angles{\nabla}^{-2} \mathbf A,
\end{equation}
where
\begin{align*}
  (A^{df})^j &= R^k(R_j A_k - R_k A_j) \, ,
  \\
  (A^{cf})^j &= -R_j R_k A^k \, .
\end{align*}

\begin{lemma}
\label{Lemma1Te}
(cf. \cite{Te},Lemma 1)
In the Lorenz gauge we have the identities

\begin{align}
\label{2.5}
[A^{\alpha},\partial_{\alpha} \phi] &= \mathcal{Q} [\Lambda^{-1} \mathbf A, \phi] + [\Lambda^{-2} A^{\alpha},\partial_{\alpha} \phi ] \, , \\
\label{2.6}
[\partial_t A^{\alpha}, \partial_{\alpha} \phi] &= Q_{0i}[A^i,\phi] \, .
\end{align}
\end{lemma}
\begin{proof}
Writing
$$ A^{\alpha} \partial_{\alpha} \phi = (-A_0 \partial_t \phi + A^{cf} \cdot \nabla \phi) + A^{df} \cdot \nabla \phi + \Lambda^{-2} \mathbf{A} \cdot \nabla \phi $$
one easily checks using the Lorenz gauge $\partial_t A_0 = \Lambda R_k A^k$ :
\begin{align*}
A^{cf} \cdot \nabla \phi &= - R_j R_k A^k \partial^j \phi = -\partial_t \Lambda^{-1} R_j A_0 \partial^j \phi \\
A_0 \partial_t \phi & = -\Lambda^{-2} \partial_j \partial^j A_0 \partial_t \phi + \Lambda^{-2} A_0 \partial_t \phi \\
& = - \partial_j(\Lambda^{-1} R^j A_0) \partial_t \phi + \Lambda^{-2} A_0 \partial_t \phi \, , 
\end{align*}
so that
$$ -A_0 \partial_t \phi + A^{cf} \cdot \nabla \phi = - Q_{0j}(\Lambda^{-1} R^j A_0,\phi) - \Lambda^{-2} A_0 \partial_t \phi \, . $$
Next
\begin{align*}
A^{df} \cdot \nabla \phi &= R^k(R_j A_k - R_k A_j)\partial^j \phi \\
&= \Lambda^{-2} \partial^k \partial_j A_k \partial^j \phi + A_j \partial^j \phi \\
&= -\half \left( \Lambda^{-2}(\partial_j \partial^j A_k - \partial_j \partial_k A^j)\partial^k \phi - \Lambda^{-2}(\partial^k \partial_j A_k - \partial_k \partial^k A_j) \partial^j \phi \right) \\
&= -\half \left(\partial_j \Lambda^{-1}(R^j  A_k - R_k A^j)\partial^k \phi - \partial^k \Lambda^{-1}(R_j A_k - R_k  A_j) \partial^j \phi \right) \\
&= - \half Q_{jk}(\Lambda^{-1}(R^j A^k - R^k A^j),\phi) \, .
\end{align*}
This leads to (\ref{2.5}). For (\ref{2.6}) we use the Lorenz gauge to obtain
\begin{align*}
[\partial_t A^{\alpha}, \partial_{\alpha} \phi] = [-\partial_t A_0,\partial_t \phi] + [\partial_t
 A^i, \partial_i \phi] 
= - [\partial_i A^i, \partial_i \phi] + [\partial_t A^i, \partial_i \phi]= Q_{0i}[A^i,\phi] \, .
\end{align*}
\end{proof}

\begin{lemma}
\label{Lemma2Te}
(cf. \cite{Te},Lemma 2)
In the Lorenz gauge the following identity holds:
$$ [A^{\alpha} , \partial_{\beta} A_{\alpha}] = \sum_{i=1}^4 \Gamma_{\beta}^i(A,\partial A,F,\partial F) \, , $$
where 
\begin{align*}
\Gamma_{\beta}^1 &= -[A_0,\partial_{\beta} A_0] + [\Lambda^{-1}R_j(\partial_t A_0), \Lambda^{-1} R^j \partial_t (\partial_{\beta} A_0)] \, , \\
\Gamma_{\beta}^2 &= \sum_{i,j} Q_{ij}[\Lambda^{-1}R_k A^k, \Lambda^{-1} R_j \partial_{\beta} A_i] + \sum_{i,j} Q_{ij}[\Lambda^{-1} R_k \partial_{\beta}A^k, \Lambda^{-1} R_j A_i] \, , \\
\Gamma_{\beta}^3 &= \sum_j \Big([\Lambda^{-1} R^i F_{ji},\Lambda^{-1} R^k \partial_{\beta} F_{jk}] + [\Lambda^{-1} R^i F_{ji}, \Lambda^{-1} \partial_{\beta} R^k[A_k,A_j]] \\
& \hspace{1em} + [\Lambda^{-1} R^i [A_i,A_j], \Lambda^{-1} \partial_{\beta} R^k F_{jk}] + [\Lambda^{-1}R^i [A_i,A_j], \Lambda^{-1} \partial_{\beta} R^k [A_k,A_j]] \Big) \, , \\
\Gamma_{\beta}^4 &= [\Lambda^{-2} \mathbf{A}, \partial_{\beta} \mathbf{A}] +[\mathbf{A}, \Lambda^{-2} \partial_{\beta} \mathbf{A}] \, .
\end{align*}
\end{lemma}
\begin{proof}
We write
$$ [A^{\alpha} , \partial_{\beta} A_{\alpha}] = -[A_0,\partial_{\beta} A_0] + [A^j,\partial_{\beta}A_j] = \sum_{i=1}^4 \Gamma_{\beta}^i(A,\partial A,F,\partial F) \, , $$
where
\begin{align*}
\Gamma_{\beta}^1 &= - [A_0,\partial_{\beta} A_0] + [A^{cf},\partial_{\beta} A^{cf}] \, , \\
\Gamma_{\beta}^2 &= [A^{cf},\partial_{\beta} A^{df}] + [A^{df},\partial_{\beta} A^{cf} ] \, , \\
\Gamma_{\beta}^3 &= [A^{df},\partial_{\beta} A^{df}] \, 
\end{align*}
and $\Gamma_{\beta}^4$ as above.

For $\Gamma_{\beta}^1$ we use $\partial_t A_0 = \Lambda R_k A^k$ and obtain
\begin{align*}
-A_0 \partial_{\beta} A_0 + A^{cf} \partial_{\beta} A^{cf} &= -A_0 \partial_{\beta}A_0 + R_j R_k A^k \partial_{\beta} R^j R_k A^k \\
&= -A_0 (\partial_{\beta} A_0) + \Lambda^{-1} R_j (\partial_t A_0)  \Lambda^{-1} \partial_t(\partial_{\beta} A_0) \, ,
\end{align*}  
which gives the result. Concerning $\Gamma_{\beta}^2$ we obtain
\begin{align*}
A^{cf} \partial_{\beta} A^{df} & = -R^j (R_k A^k) \partial_{\beta} R^i(R_j A_i-R_i A_j) \\
& = -R^j(R_k A^k) R^i(\partial_{\beta} R_j A_i) + R^i(R_k A^k) R^j (\partial_{\beta} R_j A_i) \\
&= \sum_{i,j} Q_{ij}(\Lambda^{-1}R_k A^k,\Lambda^{-1}R_j \partial_{\beta} A_i) \, ,
\end{align*}
which gives the claimed result. For $\Gamma_{\beta}^3$ we use 
$$ F_{ji} := \partial_j A_i - \partial_i A_j + [A_j,A_i] \, , $$
so that
$$ (A^{df})_j = R^i(R_j A_i - R_i A_j) = \Lambda^{-1} R^i F_{ji} + R^i(A_i A_j - A_j A_i) \, . $$
This implies
\begin{align*}
&(A^{df})^j \partial_{\beta} A^{df}_j \\
& = \sum_j \Lambda^{-1} (R^i F_{ji} + R^i(A_iA_j-A_j A_i)) \partial_{\beta} \Lambda^{-1} (R^k F_{jk} + R^k(A_k A_j - A_j A_k)) \\
& = \sum_j \Big(\Lambda^{-1}R^i F_{ji} \Lambda^{-1} R^k \partial_{\beta} F_{jk} + \Lambda^{-1} R^i F_{ji}\Lambda^{-1} \partial_{\beta} R^k ( A_k A_j - A_j A_k) \\
& \hspace{3em}+ \Lambda^{-1} R^i (A_i A_j - A_j A_i) \Lambda^{-1} \partial_{\beta} R^k F_{jk} \\
& \hspace{3em} +  \Lambda^{-1} R^i (A_i A_j - A_j A_i) \Lambda^{-1} \partial_{\beta} R^k (A_k A_j - A_j A_k) \Big) \, .
\end{align*}
Thus we obtain the claimed result.
\end{proof}

Now we refer to Tesfahun \cite{Te}, who showed that the system (\ref{YMF2}),(\ref{YM4}) in Lorenz gauge can be written in the following form by use of Lemma \ref{Lemma1Te}, (\ref{NullformTrick}) and (\ref{NewNull}) for (\ref{YMF2}), and Lemma \ref{Lemma1Te} and Lemma \ref{Lemma2Te} for (\ref{YM4}):

\begin{equation}\label{AF}
\begin{aligned}
  \square A_\beta &=  \mathcal M_\beta(A,\partial_t A,F,\partial_t F),
  \\
  \square F_{\beta\gamma} &=  \mathcal N_{\beta\gamma}(A,\partial_t A,F,\partial_t F),
\end{aligned}
\end{equation}
where

  \begin{align*}
  \mathcal M_\beta(A,\partial_t A,F,\partial_t F) &= -2 \mathcal Q[\Lambda^{-1} A,A_\beta] +
  \sum_{i=1}^4\Gamma^i_\beta(A, \partial A, F, \partial F)-2[\Lambda^{-2}  A^\alpha, \partial_\alpha A_\beta ]
  \\
 &\quad  - [A^\alpha, [A_\alpha, A_\beta]],
   \end{align*}

  \begin{align*}
  \mathcal N_{ij}(A,\partial_t A,F,\partial_t F)
  = &- 2\mathcal Q[\Lambda^{-1} A,F_{ij}]
  + 2\mathcal Q[\Lambda^{-1} \partial_j A, A_i]- 2\mathcal Q[\Lambda^{-1} \partial_i A, A_j] 
  \\
  & + 2Q_0[A_i , A_j]
  + Q_{ij}[A^\alpha,A_\alpha]-2[\Lambda^{-2}  A^\alpha, \partial_\alpha F_{ij} ]
  \\
  &+2[\Lambda^{-2}  \partial_jA^\alpha, \partial_\alpha A_{i} ]-2[\Lambda^{-2}  \partial_i A^\alpha, \partial_\alpha A_{j} ]
  \\
  & - [A^\alpha,[A_\alpha,F_{ij}]] + 2[F_{\alpha i},[A^\alpha,A_j]] - 2[F_{\alpha j},[A^\alpha,A_i]]
  \\
  & - 2[[A^\alpha,A_i],[A_\alpha,A_j]],
  \end{align*} 
  
  \begin{align*}
  \mathcal N_{0i}(A,\partial_t A,F,\partial_t F)
  = &- 2\mathcal Q[\Lambda^{-1} A,F_{0i}]
  + 2\mathcal Q[\Lambda^{-1} \partial_i A, A_0]- 2 Q_{0j}[A^j,A_i] \\
	& + 2Q_0[A_0 , A_i]
  + Q_{0i}[A^\alpha,A_\alpha]-2[\Lambda^{-2}  A^\alpha, \partial_\alpha F_{0i} ]\\
	& +2[\Lambda^{-2}  \partial_i A^\alpha, \partial_\alpha A_{0} ]
  - [A^\alpha,[A_\alpha,F_{0i}]] + 2[F_{\alpha 0},[A^\alpha,A_i]] \\
	&- 2[F_{\alpha i},[A^\alpha,A_0]]
  - 2[[A^\alpha,A_0],[A_\alpha,A_i]]
  \end{align*}
  where $\Gamma_{\beta}^i$ are defined in Lemma \ref{Lemma2Te}.

Now, looking at the terms in $ \mathcal{M}$ and $ \mathcal{N}$ and noting the fact that the Riesz transforms 
$R_i$ are bounded in the spaces involved, the estimates in Proposition \ref{Prop1.6} 
reduce to proving (we remark, that due to the multilinear character of the nonlinearity the estimates for the difference can be treated exactly like the other estimates) .\\
1. the corresponding estimates for the null forms $Q_{ij}$ , $Q_0$ and $Q \in \{Q_{0i},Q_{ij}\}$ :
\begin{align}
  \label{24}
  \norm{\Lambda_+^{-1} \Lambda_-^{\epsilon -1} Q[\Lambda^{-1} A, A]}_{F^s}
  &\lesssim \|A\|_{F^s} \|A\|_{F^s},
  \\
    \label{25}
  \norm{ \Lambda_+^{-1} \Lambda_-^{\epsilon -1} Q_{ij}[\Lambda^{-1} A, \Lambda^{-1} \partial A]}_{F^s}
  &\lesssim  \|A\|_{F^s} \|A\|_{F^s} ,
  \\
 \label{26}
  \norm{\Lambda_+^{-1} \Lambda_-^{\epsilon -1} Q[\Lambda^{-1}A, F]}_{G^r }
  &\lesssim \|A\|_{F^s}  \|F\|_{G^r},
    \\
  \label{27}
  \norm{ \Lambda_+^{-1} \Lambda_-^{\epsilon -1} Q[ A,   A]}_{G^r }
  &\lesssim \|A\|_{F^s} \|A\|_{F^s} ,\\
	\label{28}
  \norm{\Lambda_+^{-1} \Lambda_-^{\epsilon -1}  Q_0[ A,   A]}_{G^r }
  &\lesssim \|A\|_{F^s}  \|A\|_{F^s} ,
	\end{align} 
the following estimate for $\Gamma^1$ and other bilinear terms
\begin{align}
  \label{29}
  \norm{\Lambda_+^{-1} \Lambda_-^{\epsilon -1}\Gamma^1( A, \partial A)}_{F^s}
  &\lesssim\|A\|_{F^s} \|A\|_{F^s} ,
  \\
  \label{30} 
   \norm{\Lambda_+^{-1} \Lambda_-^{\epsilon -1}\Pi( A, \Lambda^{-2} \partial A  ) }_{F^s}
     &\lesssim \|A\|_{F^s} \|A\|_{F^s} ,
     \\
     \label{31} 
   \norm{\Lambda_+^{-1} \Lambda_-^{\epsilon -1} \Pi( \Lambda^{-2} A,  \partial A)   }_{F^s}
     &\lesssim \|A\|_{F^s} \|A\|_{F^s} ,
     \\
  \label{32} 
  \norm{\Lambda_+^{-1} \Lambda_-^{\epsilon -1}\Pi(\Lambda^{-1} F, \Lambda^{-1} \partial F  ) }_{F^s}
     &\lesssim \|F\|_{G^r} \|F\|_{G^r},
     \\
       \label{33} 
   \norm{\Lambda_+^{-1} \Lambda_-^{\epsilon -1} \Pi( \Lambda^{-2} A,  \partial F)   }_{G^r}
     &\lesssim \|A\|_{F^s} \|F\|_{G^r},
     \\
       \label{34} 
   \norm{ \Lambda_+^{-1} \Lambda_-^{\epsilon -1}\Pi( \Lambda^{-1} A,  \partial A)   }_{G^r}
     &\lesssim \|A\|_{F^s} \|A\|_{F^s}
     \end{align}
and\\
2. the following trilinear and quadrilinear estimates:
 \begin{align}
   \label{35}
   \norm{\Lambda_+^{-1} \Lambda_-^{\epsilon -1}\Pi(\Lambda^{-1} F,\Lambda^{-1} \partial( AA) )}_{F^s}
  &\lesssim  \|F\|_{X^{r, \half+\epsilon}} \|A\|_{F^s} \|A\|_{F^s} ,
  \\
   \label{36}
   \norm{\Lambda_+^{-1} \Lambda_-^{\epsilon -1}\Pi(\Lambda^{-1}\partial F, \Lambda^{-1}  ( AA) )}_{F^s}
  &\lesssim  \|F\|_{X^{r, \half+\epsilon}} \|A\|_{F^s} \|A\|_{F^s} ,
  \\
  \label{37}
   \norm{\Lambda_+^{-1} \Lambda_-^{\epsilon -1}\Pi(\Lambda^{-1}(AA), \Lambda^{-1} \partial ( AA)) }_{F^s}
  &\lesssim  \|A\|_{F^s} \|A\|_{F^s} \|A\|_{F^s} \|A\|_{F^s} ,
  \\
   \label{38} 
  \norm{\Lambda_+^{-1} \Lambda_-^{\epsilon -1}\Pi(A,A,A)}_{F^s}
  &\lesssim\|A\|_{F^s} \|A\|_{F^s} \|A\|_{F^s},
  \\
   \label{39}
  \norm{\Lambda_+^{-1} \Lambda_-^{\epsilon -1}\Pi(A, A, F)}_{G^r}
  &\lesssim \|A\|_{F^s} 
  \|A\|_{F^s}\|F\|_{G^r},
  \\
   \label{40}
  \norm{\Lambda_+^{-1} \Lambda_-^{\epsilon -1}\Pi(A,A, A, A)}_{G^r}
  &\lesssim  \|A\|_{F^s} \|A\|_{F^s} \|A\|_{F^s} \|A\|_{F^s},
     \end{align}
where $\Pi(\cdots)$ denotes a multilinear operator in its arguments.

The matrix commutator null forms are linear combinations of the ordinary ones, 
in view of \eqref{CommutatorNullforms}. Since the matrix 
structure plays no role in the estimates under consideration, 
we reduce (\ref{24})--(\ref{28}) to estimates to the ordinary null forms for $\mathbb C$-valued 
functions $u$ and $v$ (as in \eqref{OrdNullforms}).

The null forms above satisfy the following estimates.
\begin{lemma}
\label{Lemma2.2}
The following estimates hold for $0 \le \alpha \le 1 $ and $Q=Q_{0i}$ or $Q=Q_{ij}$:
\begin{align}
\label{41}
Q_0(u,v) &\precsim D_+^{1-\alpha} D_-^{1-\alpha} (D_+^{\alpha} u D_+^{\alpha} v) + (D_+ D_-^{1-\alpha} u)(D_+^{\alpha} v) + (D_+^{\alpha} u)(D_+ D_-^{1-\alpha} v) \\
\nonumber
Q_0(u,v) &\precsim  D_-^{1-\alpha} (D_+ u D_+^{\alpha} v) +   D_-^{1-\alpha} (D_+^{\alpha} u D_+ v)   \\
 \label{41'} &  \hspace{1em}  + (D_+ D_-^{1-\alpha} u)(D_+^{\alpha} v) + (D_+^{\alpha} u)(D_+ D_-^{1-\alpha} v) \\
\label{42}
Q(u,v) &\precsim D_+^\half D_-^\half (D_+^\half u D_+^\half v) + D_+^\half (D_+^\half D_-^\half u D_+^\half v) + D_+^\half (D_+^\half u D_+^\half D_-^\half v) \\
\nonumber
Q(u,v) &\precsim D_+^{\half-2\epsilon} D_-^{\half-2\epsilon} (D_+^{\half+2\epsilon} u D_+^{\half+2\epsilon} v)
+ D_+^{\half -2\epsilon}(D_+^{\half+2\epsilon} D_-^{\half-2\epsilon} u D_+^{\half+2\epsilon} v ) \\
\label{42'}
& \hspace{1em} + D_+^{\half -2\epsilon}(D_+^{\half+2\epsilon} u D_+^{\half+2\epsilon} D_-^{\half-2\epsilon} v) \\
\nonumber
Q(u,v) &\precsim  D_-^{\half-2\epsilon} (D_+ u D_+^{\half+2\epsilon} v) + D_-^{\half-2\epsilon} (D_+^{\half+2\epsilon} u D_+ v)
+ (D_+ D_-^{\half-2\epsilon} u)( D_+^{\half+2\epsilon} v ) \\
\nonumber
& \hspace{1em} + (D_+^{\half+2\epsilon}  u)( D_+ D_-^{\half-2\epsilon}v )
+ (D_+^{\half+2\epsilon}D_-^{\half-2\epsilon}  u)( D_+ v ) \\
& \hspace{1em}  +(D_+  u)( D_+^{\half+2\epsilon} D_-^{\half-2\epsilon} v )
\label{42''}
\end{align}
\end{lemma}
\begin{proof}
(\ref{41}) is Lemma 7.6 in \cite{KS}, and (\ref{42}) follows immediately from \cite{KMBT}, Prop. 1.
(\ref{42'}) follows by interpolating the estimate for the symbol $q=q(\tau,\xi,\lambda,\eta)$ of \cite{KMBT}, Prop. 1 which led to (\ref{42}) with its trivial bound $ q \lesssim (|\tau|+|\xi|)(|\lambda| + |\eta|) $.  (\ref{41'}) and (\ref{42''}) follow by the fractional Leibniz rule for $\Lambda_+$ and $D_+$ from (\ref{41}) and (\ref{42'}), respectively.
\end{proof}

Next we consider the term $\Gamma_{\beta}^1$ . We may ignore its matrix form and treat 
$$\Gamma^1_k(A_0,, \partial_k A_0)
=-A_0 (\partial_k A_0) + 
\Lambda^{-1} R_j (\partial_t A_0)  \Lambda^{-1} R^j  \partial_t (\partial_k A_0))$$
for $k=1,...,n$ and
\begin{align*}
\Gamma^1_0(A_0,, \partial^i A_i)
&=-A_0 (\partial_0 A_0) + 
\Lambda^{-1} R_j (\partial_t A_0)  \Lambda^{-1} R^j  \partial_t (\partial_0 A_0)) \\
& = -A_0 (\partial^i A_i) + 
\Lambda^{-1} R_j (\partial_t A_0)  \Lambda^{-1} R^j  \partial_t (\partial^i A_i)) \,,
\end{align*}
where we used the Lorenz gauge $\partial_0 A_0 = \partial^i A_i$ in the last line in order to eliminate one time derivative. Thus we have to consider
$$ \Gamma^1(u,v) = -uv + \Lambda^{-1} R_j (\partial_t u) \Lambda^{-1} R^j(\partial_t v) \, , $$
where $u=A_0$ and $v=\partial^i A_i$ or $v=\partial_k A_0$ .

The proof of the following theorem was essentially given by Tesfahun \cite{Te}. In fact the detection of this null structure was the main progress of his paper over Selberg-Tesfahun \cite{ST}.  
\begin{lemma}
\label{Lemma2.1}
The following estimates hold:
\begin{align}
\label{45'}
\Gamma^1(u,v) &\precsim \Gamma^1_1(u,v) + \Gamma^1_2(u,v) + (\Lambda^{-2}u)v + u(\Lambda^{-2}v) \, , \\
\label{45''}
\Gamma^1(u,v) &\precsim uv + \Gamma^1_2(u,v)  \, , 
\end{align}
where
\begin{align}
\nonumber
\Gamma^1_1(u,v)&= D^{\half -2\epsilon} D_-^{\half-2\epsilon} ( D^{-\half +2\epsilon} u \, D^{-\half +2\epsilon} v) + D^{\half-2\epsilon}(D^{-\half +2\epsilon} D_-^{\half -2\epsilon} u D^{-\half+2\epsilon} v ) \\
 \label{43}
& \; \; + D^{\half-2\epsilon}(D^{-\half +2\epsilon}  u \,D^{-\half+2\epsilon} D_-^{\half -2\epsilon} v ) \\ \nonumber
\Gamma^1_2(u,v) &= D_+^{\half -2\epsilon} D_-^{\half-2\epsilon} ( D_+^{\half +2\epsilon} \Lambda^{-1} u D_+^{\half +2\epsilon}\Lambda^{-1} v) + D_+ D_-^{\half-2\epsilon} \Lambda^{-1} u\, D_+^{\half+2\epsilon} \Lambda^{-1}v \\
 \label{44}
& \;\; + D_+^{\half+2\epsilon} \Lambda^{-1}  u \,D_+ D_-^{\half-2\epsilon} \Lambda^{-1} v \\ 
\nonumber
&\lesssim  D_-^{\half-2\epsilon} ( D_+ \Lambda^{-1} u D_+^{\half +2\epsilon}\Lambda^{-1} v) +
 D_-^{\half-2\epsilon} ( D_+^{\half +2\epsilon} \Lambda^{-1} u D_+ \Lambda^{-1} v) \\ 
\label{45}
& \;\; + D_+ D_-^{\half-2\epsilon} \Lambda^{-1} u\, D_+^{\half+2\epsilon} \Lambda^{-1}v
 + D_+^{\half+2\epsilon} \Lambda^{-1}  u \,D_+ D_-^{\half-2\epsilon} \Lambda^{-1} v \, .
\end{align}
\end{lemma}
\begin{proof}
$\Gamma^1(u,v)$ has the symbol
\begin{align*}
p(\xi,\tau,\eta,\lambda)& = -1 + \frac{\angles{\xi,\eta} \tau \lambda}{\angles{\xi}^2 \angles{\eta}^2} = \left( -1 + \frac{\angles{\xi,\eta} \angles{\xi,\eta}}{\angles{\xi}^2 \angles{\eta}^2} + \frac{(\tau \lambda -\angles{\xi,\eta}) \angles{\xi,\eta}}{\angles{\xi}^2 \angles{\eta}^2} \right) = I + II
\end{align*}
Now we estimate
\begin{align*}
|I| & = \left| \frac{|\xi|^2 |\eta|^2\cos^2 \angle(\xi,\eta)}{\angles{\xi}^2 \angles{\eta}^2} -1 \right| \\
&
\le \left| \frac{\angles{\xi}^2 \angles{\eta}^2 \cos^2 \angle(\xi,\eta)}{\angles{\xi}^2 \angles{\eta}^2} -1 \right|
+ \left| \frac{|\xi|^2 |\eta|^2 - \angles{\xi}^2 \angles{\eta}^2}{\angles{\xi}^2 \angles{\eta}^2} \right| \\
& = \sin^2 \angle(\xi,\eta) + \left| \frac{|\xi|^2 |\eta|^2 - \angles{\xi}^2 \angles{\eta}^2}{\angles{\xi}^2 \angles{\eta}^2} \right| \, ,
\end{align*}
where $\angle(\xi,\eta)$ denotes the angle between $\xi$ and $\eta$ .
We have
$$\left| \frac{|\xi|^2 |\eta|^2 - \angles{\xi}^2 \angles{\eta}^2}{\angles{\xi}^2 \angles{\eta}^2} \right| = \frac{|\xi|^2 + |\eta|^2 + 1}{\angles{\xi}^2 \angles{\eta}^2} \le \frac{1}{\angles{\xi}^2} + \frac{1}{\angles{\eta}^2} $$
and
\begin{align*}
&\sin^2 \angle(\xi,\eta) \le |\sin \angle(\xi,\eta)|^{1-4\epsilon} = |1-\cos \angle(\xi,\eta)|^{\half -2\epsilon} |1+\cos \angle(\xi,\eta)|^{\half - 2\epsilon} \\
& \lesssim \frac{|\xi + \eta|^{\half -2\epsilon}}{|\xi|^{\half -2\epsilon} |\eta|^{\half - 2\epsilon}} \left( ||\tau|-|\xi||^{\half-2\epsilon} + ||\lambda|-|\eta||^{\half-2\epsilon} + ||\tau + \lambda| - |\xi + \eta|^{\half-2\epsilon} \right)
\end{align*}
for $0 \le \epsilon \le \frac{1}{4}$
by \cite{KMBT}, Proof of proposition 1. Thus the operator belonging to the symbol I is controlled by $\Gamma^1_1(u,v) + (\Lambda^{-2}u)v + u (\Lambda^{-2} v)$ . 
Moreover 
$$ |II| \le \frac{|\tau \lambda - \angles{\xi,\eta}|}{\angles{\xi} \angles{\eta}} \, .$$ 
This is the symbol of $Q_0(\Lambda^{-1}u,\Lambda^{-1}v)$ , which is controlled by $\Gamma^1_2(u,v)$
by (\ref{41}). Thus we obtain (\ref{45'}) and using the trivial bound $|I| \lesssim 1$  also (\ref{45''}). Finally, (\ref{45}) follows by the fractional Leibniz rule for $D_+$ from (\ref{44}).
\end{proof}

\section{ Proof of the nonlinear estimates in the case $n \ge 4$}
\noindent
{\bf Important remark:} We assume in the following that the Fourier transforms of $u$ and $v$ are nonnegative. This means no loss of the generality, because the norms involved in the desired estimates do only depend on the size of the Fourier transforms.
\begin{proof}[Proof of (\ref{28})]
 We recall (\ref{41'}) for $\alpha = \epsilon$ :
\begin{align*}
Q_0(u,v) &\precsim  D_-^{1-\epsilon} (D_+ u D_+^{\epsilon} v) +   D_-^{1-\epsilon} (D_+^{\epsilon} u D_+ v)   \\
  &  \hspace{1em}  + (D_+ D_-^{1-\epsilon} u)(D_+^{\epsilon} v) + (D_+^{\epsilon} u)(D_+ D_-^{1-\epsilon} v) \, .
\end{align*}
Thus we have to show the following estimates and remark that we only have to consider the first and third term, because the other terms are equivalent by symmetry. \\
1. For the first term it suffices to show
\begin{align*}
 &\| \Lambda_+^{-1} \Lambda_-^{-1+\epsilon} \Lambda^{r-1} \Lambda_+  D_-^{1-\epsilon}(D_+ u D_+^{\epsilon} v) \|_{H^{0,\half+\epsilon}} \\
&\hspace{1em}\lesssim \| \Lambda^{s-1} \Lambda_+ u \|_{H^{0,\half+\epsilon}} \| \Lambda^{s-1} \Lambda_+ v \|_{H^{0,\half+\epsilon}} \,. 
\end{align*}
This follows from
$$ \|uv\|_{H^{r-1,\half+\epsilon}} \lesssim \|u\|_{H^{s-1,\half+\epsilon}} \|v\|_{H^{s-\epsilon,\half+\epsilon}} \, , $$
which is a consequence of Prop. \ref{Prop.3.8}, if $2s-r-\epsilon > \frac{n}{2}$ and $s\ge r$. This is fulfilled for a sufficiently small $\epsilon > 0$ under our assumptions. \\
2. For the third term we show
\begin{align*}
& \| \Lambda_+^{-1} \Lambda_-^{-1+\epsilon} \Lambda^{r-1} \Lambda_+ (D_+ D_-^{1-\epsilon}u D_+^{\epsilon} v) \|_{H^{0,\half+\epsilon}}
\lesssim \| \Lambda^{s-1} \Lambda_+ u \|_{H^{0,\half+\epsilon}} \| \Lambda^{s-1} \Lambda_+ v \|_{H^{0,\half+\epsilon}} \,.
\end{align*}
Thus it suffices to show
$$\|uv\|_{H^{r-1,-\half+2\epsilon}} \lesssim \|u\|_{H^{s-1,-\half+2\epsilon}} \|v\|_{H^{s-\epsilon,\half+\epsilon}} \, .$$
By duality this is equivalent to
$$\|uv\|_{H^{1-s,\half-2\epsilon}} \lesssim \|u\|_{H^{1-r,\half-2\epsilon}} \|v\|_{H^{s-\epsilon,\half+\epsilon}} \, ,$$
which is a consequence of Prop. \ref{Prop.3.8} as in 1. under the same assumptions. 
\end{proof}

\begin{proof}[Proof of (\ref{27})]
We use (\ref{42'}). By symmetry we only have to consider the first two terms.\\
1. For the first term it suffices to show
\begin{align*}
 &\| \Lambda_+^{-1} \Lambda_-^{-1+\epsilon} \Lambda^{r-1} \Lambda_+ \Lambda_+^{\half-2\epsilon} \Lambda_-^{\half-2\epsilon}(\Lambda_+^{\half+2\epsilon} u  \Lambda_+^{\half+2\epsilon} v) \|_{H^{0,\half+\epsilon}} \\
&\hspace{1em}\lesssim \| \Lambda^{s-1} \Lambda_+ u \|_{H^{0,\half+\epsilon}} \| \Lambda^{s-1} \Lambda_+ v \|_{H^{0,\half+\epsilon}} \,. 
\end{align*}
This follows from
$$ \|uv\|_{H^{r-\half-2\epsilon,0}} \lesssim \|u\|_{H^{s-\half-2\epsilon,\half+\epsilon}} \|v\|_{H^{s-\half-2\epsilon,\half+\epsilon}} \, , $$
which is a consequence of Prop. \ref{Prop.1.2} with parameters $s_0=\half-r+2\epsilon$ , $s_1 = s_2 = s-\half-2\epsilon$  , so that $s_0+s_1+s_2 > \frac{n-1}{2} $ , if $2s-r > \frac{n}{2}$ , and $s_0+s_1+s_2+s_1+s_2 > \frac{n}{2}$, which holds under our assumptions. \\
2. For the second term we show
\begin{align*}
& \| \Lambda_+^{-1} \Lambda_-^{-1+\epsilon} \Lambda^{r-1} \Lambda_+ \Lambda_+^{\half-2\epsilon} (\Lambda_+^{\half+2\epsilon} \Lambda_-^{\half-2\epsilon}u \Lambda_+^{\half+2\epsilon} v) \|_{H^{0,\half+\epsilon}} \\
&\lesssim \| \Lambda^{s-1} \Lambda_+ u \|_{H^{0,\half+\epsilon}} \| \Lambda^{s-1} \Lambda_+ v \|_{H^{0,\half+\epsilon}} \,.
\end{align*}
Using $\Lambda_-^{4\epsilon}u \precsim \Lambda_+^{4\epsilon}u$ it suffices to show
$$\|uv\|_{H^{r-\half+2\epsilon,-\half-2\epsilon}} \lesssim \|u\|_{H^{s-\half-2\epsilon,0}} \|v\|_{H^{s-\half-2\epsilon,\half+\epsilon}} \, $$
which is a consequence of Prop. \ref{Prop.1.2} as in 1. under the assumptions $2s-r > \frac{n}{2}$ and $3s-2r > \frac{n+1}{2}$ , which hold under our assumptions.
\end{proof}

\begin{proof}[Proof of (\ref{32})]
{\bf A.} We start with the first part of the $F^s$-norm.
As before it is easy to see that we can reduce to
$$ \|uv\|_{H^{s-1,-\half+2\epsilon}} \lesssim \|u\|_{H^{r+1,\half+\epsilon}} \|v\|_{H^{r,\half+\epsilon}}\,. $$
Thus we have to show the following estimates
\begin{align}
\label{48}
\|uv\|_{H^{0,-\half+2\epsilon}} &\lesssim \|u\|_{H^{r-s+2,\half+\epsilon}} \|v\|_{H^{r,\half+\epsilon}} \, ,\\ 
\label{49}
\|uv\|_{H^{0,-\half+2\epsilon}} &\lesssim \|u\|_{H^{r+1,\half+\epsilon}} \|v\|_{H^{r-s+1,\half+\epsilon}} \, .
\end{align}
1. If $2r-s > \frac{n}{2}-\frac{5}{2}$ both estimates are fulfilled by Proposition \ref{Prop.1.2} . \\
2. Assume now that $2r-s \le \frac{n}{2}-\frac{5}{2}$ and that our assumption $2r-s > \frac{n}{2} > \frac{n}{2}-3$ is fulfilled. We want to apply Corollary \ref{Cor.3.3} with the parameters $\alpha _1=r-s+2$, $\alpha_2 = r$  for  (\ref{48}), and $\alpha _1=r+1$, $\alpha_2 = r-s+1$  for  (\ref{49}). We check the conditions of that Corollary in either case.
\begin{itemize}
\item (\ref{LV2}): The condition $\alpha_1+\alpha_2 = 2r-s+2 = \frac{n}{2}-\frac{1}{q}$ can be met with a suitable $1 < q \le 2 $ under our assumption $\frac{n}{2}-\frac{5}{2} \ge 2r-s > \frac{n}{2}-3$ .
\item (\ref{LV4}): Using this $q$ the condition $\frac{1}{q} < \frac{n-1}{4}$ is equivalent to $2r-s > \frac{n}{4} - \frac{7}{4}$. If $n=4$ this means $2r-s > - \frac{3}{4}$ , whereas for $n\ge 5$ it is automatically fulfilled under our assumption $2r-s > \frac{n}{2}-3$ .
\item (\ref{LV12}): We need using (\ref{LV2}) $r+1 < \frac{n}{2}+\half - \frac{2}{q} =\frac{n}{2}+\half -n+4+4r-2s$ , which is equivalent to our assumption $3r-s > \frac{n}{2}-\frac{7}{2}$ .
\end{itemize}
{\bf B.} For the second part of the $F^s$-norm we reduce to
$$
\| \Lambda_+^{-1} \Lambda_-^{-1+\epsilon} \Lambda_+^{\half+2\epsilon} \Lambda^{-\frac{3}{4}+3\epsilon} \Lambda_-^{\half} (uv)\|_{{\mathcal L}^1_t {\mathcal L}^{4n}_x} \lesssim \|\Lambda_+ u\|_{H^{r,\half+\epsilon}} \|\Lambda_+ v\|_{H^{r-1,\half+\epsilon}} $$
and further to
$$ \| \Lambda^{-\frac{5}{4}+7\epsilon} \Lambda_-^{-\half-\epsilon} (uv) \|_{{\mathcal L}^1_t {\mathcal L}^{4n}_x} \lesssim \|u\|_{H^{r+1,\half+\epsilon}} \| v \|_{H^{r,\half+\epsilon}} \, . $$
By Prop. \ref{Prop.4.8} we obtain
$$ \| \Lambda^{-\frac{5}{4}+7\epsilon} \Lambda_-^{-\half-\epsilon} (uv) \|_{{\mathcal L}^1_t {\mathcal L}^{4n}_x} \lesssim \| \Lambda^{\frac{n}{2}-2+7\epsilon} (uv) \|_{{\mathcal L}^1_t {\mathcal L}^2_x} \, , $$
which by the fractional Leibniz rule, Sobolev and Prop. \ref{mStr} can be estimated as follows:
\begin{align*}
\| (\Lambda^{\frac{n}{2}-2+7\epsilon} u) v \|_{L^1_t L^2_x} & \lesssim \|\Lambda^{\frac{n}{2}-2+7\epsilon} u \|_{L^2_t L^{\frac{4n}{2n-7}}_x} \| v \|_{L^2_t L^{\frac{4n}{7}}_x} \\
& \lesssim \| \Lambda^{\frac{n}{2}-\frac{3}{4}+7\epsilon} u \|_{H^{0,\half+\epsilon}} \|v\|_{H^{\frac{n}{2}-\frac{7}{4},0}} \lesssim \|u\|_{H^{r+1,\half+\epsilon}} \|v\|_{H^{r,\half+\epsilon}} \, , \\
\| u \Lambda^{\frac{n}{2}-2+7\epsilon} v\|_{L^1_t L^2_x} & \lesssim \|u\|_{L^2_t L^{4n}_x} \| \Lambda^{\frac{n}{2}-2+7\epsilon} v \|_{L^2_t L^{\frac{4n}{2n-1}}_x} \\
& \lesssim \|u\|_{H^{\frac{n}{2}-\frac{3}{4},\half+\epsilon}} \| v \|_{H^{\frac{n}{2}-\frac{7}{4}+7\epsilon,0}}
 \lesssim \|u\|_{H^{r+1,\half+\epsilon}} \|v\|_{H^{r,\half+\epsilon}} \, .
\end{align*}
\end{proof}

\begin{proof}[Proof of (\ref{29})]
{\bf A.} For the first part of the $F^s$-norm it is sufficient to show
\begin{equation}
\label{G1}
\| \Gamma^1(u,v) \|_{H^{s-1,-\half+2\epsilon}} \lesssim \|u\|_{F^s} \|\Lambda_+ v\|_{H^{s-2,\half+\epsilon}} \, . 
\end{equation}\
We use Lemma \ref{Lemma2.1}. \\
a. We first consider $\Gamma_2^1(u,v)$ . By (\ref{45}) it suffices to show the following estimates, all of which are consequences of Proposition \ref{Prop.1.2} and Corollary \ref{Cor.1.1} . \\
$$
\|D_+ \Lambda^{-1} u D_+^{\half+2\epsilon} \Lambda^{-1} v \|_{H^{s-1,0}}\ \lesssim \|\Lambda_+ u\|_{H^{s-1,\half+\epsilon}}  \|\Lambda_+ v\|_{H^{s-2,\half+\epsilon}}  , ,$$
which follows from
$$ \|uv\|_{H^{s-1,0}} \lesssim \|u\|_{H^{s,\half+\epsilon}} \|v\|_{H^{s-\half-2\epsilon,\half+\epsilon}}$$
and
\begin{align*}
\|uv\|_{H^{s-1,0}} & \lesssim \|u\|_{H^{s+\half-2\epsilon,\half+\epsilon}} \|v\|_{H^{s-1,\half+\epsilon}} \, ,\\
\|uv\|_{H^{s-1,-\half+2\epsilon}} & \lesssim \|u\|_{H^{s,0}} \|v\|_{H^{s-\half-2\epsilon,\half+\epsilon}} \, , \\
\|uv\|_{H^{s-1,-\half+2\epsilon}} & \lesssim \|u\|_{H^{s+\half-2\epsilon,\half+\epsilon}} \|v\|_{H^{s-1,0}} \, .
\end{align*}
b.  Assume that $u$ and $v$ have frequencies $ \ge 1$ ,  so that $\Lambda_+^{\alpha} u \sim D_+^{\alpha}u$. In this case we use (\ref{45'}) and consider $\Gamma^1_1(u,v)$ . By (\ref{43}) we may reduce estimates for the  first and third terms on the right hand side to 
\begin{align*}
\|uv\|_{H^{s-\half-2\epsilon,0}} & \lesssim \|u\|_{H^{s+\half-2\epsilon,\half+\epsilon}} 
\|v\|_{H^{s-\half-2\epsilon,\half+\epsilon}} \, , \\
\|uv\|_{H^{s-\half-2\epsilon,-\half+2\epsilon}} & \lesssim \|u\|_{H^{s+\half-2\epsilon,\half+\epsilon}} 
\|v\|_{H^{s-\half-2\epsilon,0}} \, .
\end{align*}
Both estimates follow from Corollary \ref{Cor.1.1} .  The second term is reduced to the following estimate
\begin{align}
\label{R1}
  \|uv\|_{H^{s-\half-2\epsilon,-\half+2\epsilon}} \lesssim \|\Lambda^{\half-2\epsilon} \Lambda_-^{-\half+2\epsilon} u \|_{F^s} \|v \|_{H^{s-\half-2\epsilon,\half + \epsilon}} \, .
\end{align}
By the fractional Leibniz rule we have to show the following two estimates:\\
b1. $$\|(\Lambda^{s-\half-2\epsilon}u)v\|_{H^{0,-\half+2\epsilon}} \lesssim 
\|u \|_{H^{s+\half-2\epsilon,0}} \|v \|_{H^{s-\half-2\epsilon,\half + \epsilon}} \, ,$$
which follows from Corollary \ref{Cor.1.1}. \\
b2. $$\|u( \Lambda^{s-\half-2\epsilon}v)\|_{H^{0,-\half+2\epsilon}} \lesssim \|\Lambda^{\half-2\epsilon} \Lambda_-^{-\half + 2\epsilon} u \|_{F^s} \|v \|_{H^{s-\half-2\epsilon,\half + \epsilon}} \, .$$
Using the definition of the $F^s$-norm this is implied by  
\begin{equation}
\label{*}
\|uv\|_{H^{0,-\half+2\epsilon}} \lesssim (\|\Lambda^{\half-2\epsilon}\Lambda^{-\frac{1}{4}+5\epsilon} \Lambda_-^{2\epsilon}u\|_{{\mathcal L}_t^1 {\mathcal L}_x^{4n}} + \|\Lambda^{\half-2\epsilon} u\|_{H^{s,0}})\|v\|_{H^{0,\half+\epsilon}} \, .
\end{equation}
By Lemma \ref{Lemma8.10} we have for $\alpha,\beta \ge 0$ :
$$\|uv\|_{H^{0,-\beta}} \lesssim \|(\Lambda^{-\beta} u) v\|_{H^{0,0}} + \|(\Lambda^{\alpha} u) v\|_{H^{0,-\beta-\alpha}}$$for $u,v$ with $\widehat{u},\widehat{v} \ge 0$ . Thus for the choice $\beta=\half-2\epsilon$ , $\alpha= \frac{5}{2}\epsilon$ the estimate (\ref{*}) reduces to the following two estimates:\\
b2.1.
 \begin{align*}
\|(\Lambda^{-\half+2\epsilon} u)v\|_{H^{0,0}} & \lesssim \|\Lambda^{\half-\epsilon} u\|_{H^{s,0}} \|v\|_{H^{0,\half + \epsilon}} \\
\Longleftrightarrow \quad \|uv\|_{H^{0,0}} & \lesssim \|u\|_{H^{s+1-3\epsilon,0}} \|v\|_{H^{0,\half + \epsilon}} \, .
\end{align*}
This follows by Sobolev, because $s > \frac{n}{2} - 1$ . \\
b2.2.
 \begin{align*}
 \|(\Lambda^{\frac{5}{2}\epsilon} u) v \|_{H^{0,-\half-\epsilon}} &\lesssim \| \Lambda^{\frac{1}{4}+3\epsilon} u \|_{{\mathcal L}_t^1 {\mathcal L}_x^{4n}} \|v\|_{H^{0,\half+\epsilon}} \\
\Longleftrightarrow \quad \|uv\|_{H^{0,-\half-\epsilon}} & \lesssim  \| \Lambda^{\frac{1}{4}+\frac{\epsilon}{2}} u \|_{{\mathcal L}_t^1 {\mathcal L}_x^{4n}} \|v\|_{H^{0,\half+\epsilon}} \, .
\end{align*}
This follows from Prop. \ref{Cor.4.6}  and Prop. \ref{Prop.4.3}:
$$ \|uv\|_{H^{0,-\half-\epsilon}} \lesssim \|uv\|_{{\mathcal L}_t^1 {\mathcal L}_x^2} \lesssim \|u\|_{{\mathcal L}_t^1 {\mathcal L}_x^{\infty}} \|v\|_{L^{\infty}_t L^2_x} \lesssim \| \Lambda^{\frac{1}{4}+\frac{\epsilon}{2}} u \|_{{\mathcal L}_t^1 {\mathcal L}_x^{4n}} \|u\|_{H^{0,\half+\epsilon}} \, . $$
c. Consider $(\Lambda^{-2}u) v$ and $u (\Lambda^{-2}v)$ . It suffices to show
\begin{align*}
\| (\Lambda^{-2}  u)v\|_{H^{s-1,-\half+2\epsilon}} &\lesssim \|u\|_{H^{s,\half + \epsilon}} \|v\|_{H^{s-1,\half+\epsilon}} \, , \\
\|   u(\Lambda^{-2}v)\|_{H^{s-1,-\half+2\epsilon}} &\lesssim \|u\|_{H^{s,\half + \epsilon}} \|v\|_{H^{s-1,\half+\epsilon}}\, .
\end{align*}
Both follow easily from Prop. \ref{SML} under our asumption $s > \frac{n}{2}-1$ . \\
d. Let us finally consider the case where the frequencies of $u$ or $v$ are $\le 1$. We use (\ref{45''}) instead of (\ref{45'}). Because $\Gamma^1_2(u,v)$ has already been handled, we only have to consider $uv$ . If $u$ has low frequencies we obtain by Sobolev's multiplication law  (\ref{SML}):
\begin{align*}
\| uv\|_{H^{s-1,-\half+2\epsilon}} & \lesssim \|u\|_{H^{\frac{n}{2}+,\half+\epsilon}} \|v\|_{H^{s-1,\half+\epsilon}} \\
& \lesssim \|u\|_{H^{s,\half+\epsilon}} \|v\|_{H^{s-1,\half+\epsilon}} \, .
\end{align*}
Similarly we treat the case where $v$ has low frequencies.
\\[1em]
{\bf B.} Now we consider the second part of the $F^s$-norm. We want to show
\begin{equation}
\label{G2}
 \|\Lambda_+^{-1} \Lambda_-^{-1+\epsilon} \Lambda_+^{\half+2\epsilon} \Lambda^{-\frac{3}{4}+3\epsilon} \Lambda_-^{\half} \Gamma^1(u,v)\|_{{\mathcal L}_t^1 {\mathcal L}_x^{4n}} \lesssim \|u\|_{F^s} \| \Lambda_+ v\|_{H^{s-2,\half+\epsilon}} \, . 
\end{equation}
We use Lemma \ref{Lemma2.1} . \\
a. We first consider $\Gamma^1_2(u,v)$ . \\
1. The estimate for the first term on the right hand side of (\ref{44}) reduces to
\begin{align*}
& \|\Lambda_+^{-1} \Lambda_-^{-1+\epsilon} \Lambda_+^{\half+2\epsilon} \Lambda^{-\frac{3}{4}+3\epsilon} \Lambda_-^{\half} D_+^{\half-2\epsilon} D_-^{\half-2\epsilon} ( D_+^{\half+2\epsilon} \Lambda^{-1} u D_+^{\half+2\epsilon} \Lambda^{-1} v)\|_{{\mathcal L}_t^1 {\mathcal L}_x^{4n}} \\& \hspace{1em}\lesssim \|u\|_{F^s} \| \Lambda_+ v\|_{H^{s-2,\half+\epsilon}}  
\end{align*}
and therefore we only have to prove
\begin{equation}
\label{101}
\| \Lambda^{-\frac{3}{4}+3\epsilon} (uv)\|_{{\mathcal L}_t^1 {\mathcal L}_x^{4n}} \lesssim \|u\|_{H^{s+\half-2\epsilon,\half+\epsilon}} \| \Lambda_+ v\|_{H^{s-\half-2\epsilon,\half+\epsilon}} \, . 
\end{equation}
This follows from Proposition \ref{TheoremE} with parameters $q=2$ , $r=8n$ , $\sigma = \frac{3}{4}-3\epsilon$,  $s_1=\frac{n}{2}-\frac{5}{8}-\epsilon$ and $s_2=\frac{n}{2}-\frac{11}{8}+5\epsilon $ . The claimed estimate follows, because $s_1 < s+\half-2\epsilon$ and $s_2 < s-\frac{1}{2}-2\epsilon$ under our assumption $s > \frac{n}{2}-\frac{7}{8}$ . \\
2. For the second term we have to show
$$ \| \Lambda^{-\frac{5}{4}+5\epsilon} \Lambda_-^{-\half+\epsilon} (uv)\|_{{\mathcal L}_t^1 {\mathcal L}_x^{4n}} \lesssim \|u\|_{H^{s,0}} \|v\|_{H^{s-\half-2\epsilon,\half + \epsilon}} \, , $$
which by Lemma \ref{Lemma8.10} reduces to
\begin{align*}
\| \Lambda^{-\frac{5}{4}+5\epsilon} \Lambda_-^{-\half-\epsilon} (uv)\|_{{\mathcal L}_t^1 {\mathcal L}_x^{4n}} &\lesssim \|u\|_{H^{s-2\epsilon,0}} \|v\|_{H^{s-\half-4\epsilon,\half + \epsilon}} \, .
\end{align*}
 We obtain by Proposition \ref{Prop.4.8}
$$ \| \Lambda^{-\frac{5}{4}+5\epsilon} \Lambda_-^{-\half-\epsilon}
(uv)\|_{{\mathcal L}^1_t {\mathcal L}^{4n}_x} 
\lesssim \|\Lambda^{\frac{n}{2}-2+5\epsilon}
(uv)\|_{{\mathcal L}^1_t {\mathcal L}^2_x}  \ , 
$$
so that by the fractional Leibniz rule this requires the following estimates: \\
2.1. $$\|u \Lambda^{\frac{n}{2}-2+5\epsilon} v\|_{L^1_t L^2_x} \lesssim \|u\|_{L^2_t L^n_x} \|\Lambda^{\frac{n}{2}-2+5\epsilon} v\|_{L^2_t L^{\frac{2n}{n-2}}_x} \, .$$  
By Sobolev the first factor is bounded by $\|u\|_{H^{s-2\epsilon,0}}$ for $s> \frac{n}{2}-1$. For the second factor we use Proposition \ref{mStr}, which gives 
\begin{equation}
\label{***}
\|v\|_{L^2_t L^{\frac{2n}{n-2}}_x} \lesssim \|v\|_{H^{\half + \frac{1}{2n},\half+\epsilon}} \, , 
\end{equation}
thus
 $$                                                                                       \|\Lambda^{\frac{n}{2}-2+5\epsilon} v\|_{L^2_t L^{\frac{2n}{n-2}}_x} \lesssim \|v\|_{H^{\frac{n}{2}-\frac{3}{2}+\frac{1}{2n}+5\epsilon,\half+\epsilon}} \lesssim \|v\|_{H^{s-\frac{1}{2}-4\epsilon,\half+\epsilon}} \, , $$             
because $s> \frac{n}{2}-\frac{7}{8}$ and $n \ge 4$ . \\
2.2. $$\| (\Lambda^{\frac{n}{2}-2+5\epsilon}u) v\|_{L^1_t L^2_x} \lesssim \|\Lambda^{\frac{n}{2}-2+5\epsilon}u\|_{L^2_t L^{\frac{2n}{n-2}_x}} \| v\|_{L^2_t L^n_x} \, .$$ 
The first factor is bounded by Sobolev by $\|u\|_{H^{s-2\epsilon,0}}$ for $s > \frac{n}{2} -1$ and the second factor in the case $n \ge 5$ by Proposition \ref{Str} 
$$ \|v\|_{L^2_t L^n_x} \lesssim \|v\|_{H^{\frac{n}{2}-\frac{3}{2},\half+\epsilon}} \lesssim \|v\|_{H^{s-\half-4\epsilon,\half+\epsilon}} \, , $$
whereas in the case $n=4$ we have to use (\ref{***}) , which implies
$$ \|v\|_{L^2_t L^4_x} \lesssim \|v\|_{H^{\frac{5}{8},\half+\epsilon}} \le \|v\|_{H^{s-\half-4\epsilon,\half+\epsilon}}$$
under our assumption $s > \frac{9}{8}$ .   \\
3. $$ \|\Lambda^{-\frac{5}{4}+5\epsilon} \Lambda_-^{-\half+\epsilon}(uv)\|_{{\mathcal L}_t^1 {\mathcal L}_x^{4n}} \lesssim \|u\|_{H^{s+\half-2\epsilon,\half+\epsilon}} \|v\|_{H^{s-1,0}} \, . $$                   
By Lemma \ref{Lemma8.10} this reduces to the following estimate:
\begin{align}
\label{55}
\|\Lambda^{-\frac{5}{4}+5\epsilon} \Lambda_-^{-\half-\epsilon}(uv)\|_{{\mathcal L}_t^1 {\mathcal L}_x^{4n}} \lesssim \|u\|_{H^{s+\half-4\epsilon,\half+\epsilon}} \|v\|_{H^{s-1-4\epsilon,0}} \, .
\end{align}
We obtain by Proposition \ref{Prop.4.8} 
$$ \|\Lambda^{-\frac{5}{4}+5\epsilon} \Lambda_-^{-\half-\epsilon}(uv)\|_{{\mathcal L}_t^1 {\mathcal L}_x^{4n}} \lesssim \| \Lambda^{\frac{n}{2}-2+5\epsilon} (uv) \|_{{\mathcal L}_t^1 {\mathcal L}_x^2} \ , $$
 so that by Leibniz' rule we argue as follows: \\
3.1. By Sobolev and Strichartz (Proposition \ref{Str}) we estimate
\begin{align*}
\|(\Lambda^{\frac{n}{2}-2+5\epsilon} u) v \|_{L^1_t L^2_x} &\lesssim \|\Lambda^{\frac{n}{2}-2+5\epsilon} u \|_{L^2_t L^{\frac{2n}{n-4}}} \|v\|_{L^2_t L^{\frac{n}{2}}_x} \\
& \lesssim \|u\|_{H^{\frac{n}{2}-\half +5\epsilon,\half+2\epsilon}} \|v\|_{H^{\frac{n}{2}-2,0}} \, ,
\end{align*}
so that the desired estimate follows for $s>\frac{n}{2}-1$ . \\
3.2. Using Proposition \ref{Str} again we also obtain
$$ \| u \Lambda^{\frac{n}{2}-2+5\epsilon} v \|_{L^1_t L^2_x} \lesssim \|u\|_{L^2_t L^{\infty}_x} \|\Lambda^{\frac{n}{2}-2+5\epsilon} v \|_{L^2_t L^2_x} \lesssim \|u\|_{H^{\frac{n}{2}-\half+\epsilon,\half + \epsilon}} \|v\|_{H^{\frac{n}{2}-2+5\epsilon,0}} \ , $$
which is sufficient for $s> \frac{n}{2}-1$ . \\
b1. In the case where $u$ and $v$ have frequencies $ \ge 1$ we use (\ref{45'}) and consider $\Gamma_1^1$ .  \\
1. The estimate for the first term on the right hand side of (\ref{43}) reduces to (\ref{101}). \\
2. The second term on the right hand side of (\ref{43}) reduces to
\begin{align*}
 &\|\Lambda^{-\frac{3}{4}+3\epsilon} \Lambda_-^{-\half+2\epsilon} (\Lambda^{-\half+2\epsilon} u \Lambda^{-\half+2\epsilon} v) \|_{{\mathcal L}_t^1 {\mathcal L}_x^{4n}} \\
&\vspace{1em} \lesssim (\|u\|_{H^{s,0}} + \|\Lambda^{-\frac{1}{4}+5\epsilon} u\|_{{\mathcal L}_t^1 {\mathcal L}_x^{4n}} ) \|v\|_{H^{s-1,\half+\epsilon}} \, .
\end{align*}
By Lemma \ref{Lemma8.10} this requires 
\begin{align}
\nonumber
&\|\Lambda^{-\frac{3}{4}+3\epsilon} \Lambda_-^{-\half-2\epsilon} (\Lambda^{-\half+2\epsilon} u \Lambda^{-\half+2\epsilon} v) \|_{{\mathcal L}_t^1 {\mathcal L}_x^{4n}} \\
\label{57}
&\vspace{1em} \lesssim (\|u\|_{H^{s-2\epsilon,0}} + \|\Lambda^{-\frac{1}{4}+3\epsilon} u\|_{{\mathcal L}_t^1 {\mathcal L}_x^{4n}} ) \|v\|_{H^{s-1-2\epsilon,\half+\epsilon}} \, .
\end{align}
We start with Proposition \ref{Prop.4.8} 
\begin{align}
\label{58'}
\|\Lambda^{-\frac{3}{4}+3\epsilon} \Lambda_-^{-\half-2\epsilon} (\Lambda^{-\half+2\epsilon} u \Lambda^{-\half+2\epsilon} v) \|_{{\mathcal L}_t^1 {\mathcal L}_x^{4n}} \lesssim \|\Lambda^{-\frac{n}{2}-\frac{3}{2}+3\epsilon} (\Lambda^{-\half+2\epsilon} u \Lambda^{-\half+2\epsilon} v) \|_{{\mathcal L}_t^1 {\mathcal L}_x^2}
\end{align}
and use the fractional Leibniz rule to reduce to \\
2.1.
\begin{align*}
\|(\Lambda^{\frac{n}{2}-\frac{3}{2}+3\epsilon} \Lambda^{-\half+2\epsilon} u) (\Lambda^{-\half+2\epsilon} v) \|_{L_t^1  L_x^2}& \lesssim \|\Lambda^{\frac{n}{2}-\frac{3}{2}+3\epsilon} \Lambda^{-\half+2\epsilon} u\|_{L^2_t L^{\frac{2n}{n-2}}_x} \| \Lambda^{-\half+2\epsilon} v \|_{L_t^2  L_x^n} \\
&\lesssim \|\Lambda^{\frac{n}{2}-\frac{3}{2}+3\epsilon} \Lambda^{\half+2\epsilon} u\|_{L^2_t L^2_x} \| v \|_{H^{\frac{n}{2}-2+2\epsilon,\half+\epsilon}} \\
&\lesssim \|u\|_{H^{s-2\epsilon,0}}\ \|v\|_{H^{s-1-2\epsilon,\half+\epsilon}} \, .
\end{align*}
This is implied by Sobolev and a direct application of Proposition \ref{Str}, which is possible for $n\ge 5$, as one easily checks. In the case $n=4$ we use Prop. \ref{mStr}, which gives
$$\|\Lambda^{-\half+2\epsilon} v \|_{L^2_t L^4_x} \lesssim \|v\|_{H^{\frac{1}{8}+2\epsilon,\half+\epsilon}} \le \|v\|_{H^{s-1-2\epsilon,\half+\epsilon}}$$
under our assumption $s > \frac{9}{8}$ .\\
2.2. By use of Prop. \ref{Prop.4.3} for the first step and Cor. \ref{Cor.4.6} for the second step we obtain
\begin{align*}
&\|(\Lambda^{-\half+2\epsilon} u) (\Lambda^{-\frac{n}{2}-\frac{3}{2}+3\epsilon}  \Lambda^{-\half+2\epsilon} v) \|_{{\mathcal L}_t^1 {\mathcal L}_x^2} \\
&\lesssim \|\Lambda^{-\half+2\epsilon} u \|_{{\mathcal L}^1_t {\mathcal L}^{\infty}_x} \| \Lambda^{\frac{n}{2}-\frac{3}{2}+3\epsilon} \Lambda^{-\half+2\epsilon} v \|_{L^{\infty}_t L^2_x} \\
&\lesssim \| \Lambda^{-\half+3\epsilon} u \|_{{\mathcal L}^1_t {\mathcal L}^{4n}_x} \|v\|_{H^{\frac{n}{2}-2+5\epsilon,\half+\epsilon}} \\
&\lesssim \| \Lambda^{-\half+5\epsilon} u \|_{{\mathcal L}^1_t {\mathcal L}^{4n}_x} \|v\|_{H^{s-1-2\epsilon,\half+\epsilon}} \, .
\end{align*}
3. The estimate (\ref{57}) has to be replaced by
\begin{align}
\label{59}
\|\Lambda^{-\frac{3}{4}+3\epsilon} \Lambda_-^{-\half-2\epsilon} (\Lambda^{-\half+2\epsilon} u \Lambda^{-\half+2\epsilon} v) \|_{{\mathcal L}_t^1 {\mathcal L}_x^{4n}} &
\lesssim \|u\|_{H^{s-2\epsilon,\half+\epsilon}}  \|v\|_{H^{s-1-2\epsilon,0}} \, .
\end{align}
 We argue similarly as for 2. We first apply (\ref{58'}) and reduce to the following estimates: \\
3.1. By Sobolev and Prop.\ref{Str} we obtain
\begin{align*}
&\|(\Lambda^{\frac{n}{2}-\frac{3}{2}+3\epsilon} \Lambda^{-\half+2\epsilon} u) (\Lambda^{-\half+2\epsilon} v) \|_{L_t^1  L_x^2} \\
& \lesssim \|\Lambda^{\frac{n}{2}-\frac{3}{2}+3\epsilon} \Lambda^{-\half+2\epsilon} u\|_{L^2_t L^{\frac{2n}{n-3}}_x} \| \Lambda^{-\half+2\epsilon} v \|_{L_t^2  L_x^{\frac{2n}{3}}} \\
&\lesssim \| u\|_{H^{\frac{n}{2}-1+3\epsilon,\half+\epsilon}} \| v \|_{H^{\frac{n}{2}-2+2\epsilon,0}} \\
&\lesssim \|u\|_{H^{s-2\epsilon,\half+\epsilon}}\ \|v\|_{H^{s-1-2\epsilon,0}} \, .
\end{align*}
3.2. Similarly
\begin{align*}
&\| \Lambda^{-\half+2\epsilon} u (\Lambda^{\frac{n}{2}-\frac{3}{2}+3\epsilon}\Lambda^{-\half+2\epsilon} v) \|_{L_t^1  L_x^2} \\
& \lesssim \| \Lambda^{-\half+2\epsilon} u\|_{L^2_t L^{\infty}_x} \| \Lambda^{\frac{n}{2}-\frac{3}{2}+3\epsilon} \Lambda^{-\half+2\epsilon} v \|_{L_t^2  L_x^2} \\
&\lesssim \|\Lambda^{-\half+2\epsilon} u\|_{H^{\frac{n}{2}-\half+\epsilon,\half+\epsilon}} \| \Lambda^{-\half+2\epsilon} v \|_{H^{\frac{n}{2}-\frac{3}{2}+3\epsilon,0}} \\
&\lesssim \|u\|_{H^{s-2\epsilon,\half+\epsilon}}\ \|v\|_{H^{s-1,0}} \, .
\end{align*}
b2. Now we consider the case where $u$ and/or $v$ have frequency $\le 1$ and use (\ref{45''}). It remains to consider the first term $uv$ . Thus it suffices to show
$$ \|\Lambda_+^{-1} \Lambda_-^{-1+\epsilon} \Lambda_+^{\half+2\epsilon} \Lambda^{-\frac{3}{4}+3\epsilon} \Lambda_-^{\half} (uv)\|_{{\mathcal L}^1_t {\mathcal L}^{4n}_x} \lesssim \|u\|_{H^{s,\half+\epsilon}} \|v\|_{H^{s-1,\half+\epsilon}} \, . $$
We crudely estimate the left hand side by
$$ \| \Lambda	^{-\frac{5}{4}+5\epsilon} (uv)\|_{L^1_t L^{4n}_x} \lesssim \|u\|_{H^{s_1,\half+\epsilon}} \|v\|_{H^{s_2,\half+\epsilon}} $$
by use of  Proposition \ref{TheoremE}, where $s_1$ and $s_2$ have to fulfill $s_1+s_2 \ge n-\frac{5}{2}+5\epsilon$ and $s_1,s_2 < \frac{n}{2}-\frac{5}{8}$ . If u has frequency $\le 1$ choose $s_1= \frac{n}{2}-\frac{5}{8}-\epsilon$ , $s_2= \frac{n}{2}-\frac{15}{8}+6\epsilon < s-1.$  The value of $s_1$ is irrelevant in view of the low frequency assumption. If $v$ has frequency $\le 1$ choose $s_1= \frac{n}{2}-1 < s$ , $s_2= \frac{n}{2}-\frac{3}{2}+5\epsilon$ , where the value of $s_2$ is irrelevant. \\
c. According to Lemma \ref{Lemma2.1} we finally have to consider
 $(\Lambda^{-2}u)v$ and $u(\Lambda^{-2}v)$ . We have to show
\begin{align*}
\|\Lambda_+^{-1} \Lambda_-^{-1+\epsilon} \Lambda_+^{\half+2\epsilon} \Lambda^{-\frac{3}{4}+3\epsilon} \Lambda_-^{\half} ((\Lambda^{-2}u)v)\|_{{\mathcal L}^1_t {\mathcal L}^{4n}_x} 
 \lesssim \|\Lambda^{-2}u\|_{H^{s+2,\half+\epsilon}} \|v\|_{H^{s-1,\half+\epsilon}} \, . 
\end{align*}
We argue as in 4. choosing $s_1=\frac{n}{2}-\frac{5}{8}-\epsilon < s+2$ and $s_2=\frac{n}{2}-\frac{15}{8}+6\epsilon < s-1$ . In the same way we also obtain
\begin{align*}
\|\Lambda_+^{-1} \Lambda_-^{-1+\epsilon} \Lambda^{\half+2\epsilon}\Lambda^{-\frac{3}{4}+3\epsilon} \Lambda_-^{\half} (u(\Lambda^{-2}v))\|_{{\mathcal L}^1_t {\mathcal L}^{4n}_x} 
 \lesssim \|u\|_{H^{s,\half+\epsilon}} \|\Lambda^{-2}v\|_{H^{s+1,\half+\epsilon}} 
\end{align*}
by the choice $s_1 = \frac{n}{2}-1 < s$ , $s_2=\frac{n}{2}-\frac{3}{2}+5\epsilon < s+1$ . The proof of (\ref{29}) is now complete.
\end{proof}

\begin{proof}[Proof of (\ref{24}) and (\ref{25})] 
We have to prove
$$ \| \Lambda_+^{-1} \Lambda_-^{-1+\epsilon} Q(u,v) \| _{F^s} \lesssim \|\Lambda u \|_{F^s} \|\Lambda_+ v\|_{H^{s-1,\half+\epsilon}} \, .$$
We want to use (\ref{42'}). \\
{\bf A.} The first part of the $F^s$-norm is handled as follows:\\
1. For the first term we reduce to the estimate
$$ \|uv\|_{H^{s-\half-2\epsilon,0}} \lesssim \|u\|_{H^{s+\half-2\epsilon,\half +\epsilon}} \| v\|_{H^{s-\half-2\epsilon,\half+\epsilon}} \, 
, $$
which follows by Prop. \ref{Prop.1.2}.
 \\
2. For the last term we reduce to
\begin{align*}
\|uv\|_{H^{s-\half-2\epsilon,-\half+2\epsilon}} 
 \lesssim \| u\|_{H^{s+\half-2\epsilon,\half+\epsilon}} \|v \|_{H^{s-\half-2\epsilon,0}} \, ,
\end{align*}
which holds by Prop. \ref{Prop.1.2}. \\
3. For the second term we want to prove
\begin{align}
\label{R2}
\|uv\|_{H^{s-\half-2\epsilon,-\half+2\epsilon}} 
 \lesssim   \|\Lambda^{\half-2\epsilon} \Lambda_-^{-\half+2\epsilon} u\|_{F^s} \|v \|_{H^{s-\half-2\epsilon,\half+\epsilon}} \, ,
\end{align}
which follows from
\begin{align*}
\|uv\|_{H^{s-\half-2\epsilon,-\half+2\epsilon}} 
 \lesssim (\|u\|_{H^{s+\half-2\epsilon,0}} + \|\Lambda^{\frac{1}{4}+3\epsilon} u\|_{{\mathcal L}^1_t {\mathcal L}^{4n}_x}) \|v \|_{H^{s-\half-2\epsilon,\half+\epsilon}} \, ,
\end{align*}
By Lemma \ref{Lemma8.10} this reduces to the following two estimates:
\begin{align*}
\|uv\|_{H^{s-\half-2\epsilon,-\half-\frac{\epsilon}{2}}} &
 \lesssim (\|u\|_{H^{s+\half-\frac{9}{2}\epsilon,0}} + \|\Lambda^{\frac{1}{4}+\frac{\epsilon}{2}} u\|_{{\mathcal L}^1_t {\mathcal L}^{4n}_x}) \|v \|_{H^{s-\half-2\epsilon,\half+\epsilon}} \, , \\
\|uv\|_{H^{s-\half-2\epsilon,0}} 
 &\lesssim \|u\|_{H^{s+1-4\epsilon,0}}  \|v \|_{H^{s-\half-2\epsilon,\half+\epsilon}} \, .
\end{align*}
The last estimate follows immediately from Prop. \ref{SML}. The first estimate is handled by the fractional Leibniz rule:\\
3.1. By Sobolev we have
\begin{align*}
\|u \Lambda^{s-\half-2\epsilon} v\|_{H^{0,-\half-\frac{\epsilon}{2}}} & \lesssim \|u \Lambda^{s-\half-2\epsilon} v\|_{{\mathcal L}^1_t {\mathcal L}^2_x} \\
& \lesssim \|u\|_{{\mathcal L}^1_t {\mathcal L}^{\infty}_x} \|\Lambda^{s-\half-2\epsilon} v\|_{L^{\infty}_t L^2_x} \\
& \lesssim \| \Lambda^{\frac{1}{4}+\frac{\epsilon}{2}} u \|_{{\mathcal L}^1_t {\mathcal L}^{4n}_x} \|v\|_{H^{s-\half-2\epsilon,\half+\epsilon}}  \, .
\end{align*}
3.2.  We obtain
\begin{align*}
\|(\Lambda^{s-\half-2\epsilon}u)  v\|_{H^{0,-\half-\frac{\epsilon}{2}}} &  \lesssim \|u\|_{H^{s+\half-\frac{9}{2}\epsilon,0}} \|v \|_{H^{s-\half-2\epsilon,\half+\epsilon}} 
\end{align*}
by Prop. \ref{Prop.1.2} with parameters $s_0=1-\frac{5}{2}\epsilon$ , $s_1=0$ , $s_2=s-\half-2\epsilon$ , so that $s_0+s_1+s_2=s+\half-\frac{9}{2}\epsilon >  \frac{n-1}{2}$ and $(s_0+s_1+s_2)+s_1+s_2 = 2s-\frac{13}{2}\epsilon > \frac{n}{2} $ , because $s > \frac{n}{2}-\frac{7}{8} > \frac{n}{4} $ for $n \ge 4$ .
\\
{\bf B.} The second part of the $F^s$-norm is handled as follows:\\
1. The first term on the right hand side of (\ref{42'}) requires the estimate
\begin{align*}
& \|\Lambda_+^{-1} \Lambda_-^{-1+\epsilon} \Lambda_+^{\half+2\epsilon} \Lambda^{-\frac{3}{4}+3\epsilon} \Lambda_-^{\half}\Lambda_+^{\half-2\epsilon} \Lambda_-^{\half-2\epsilon}(\Lambda_+^{\half+2\epsilon} u \Lambda_+^{\half+2\epsilon} v)\|_{{\mathcal L}^1_t {\mathcal L}^{4n}_x} \\
& \hspace{1em} \lesssim \| \Lambda_+ u\|_{H^{s,\half+\epsilon}} \|\Lambda_+ v \|_{H^{s-1,\half+\epsilon}} \, ,
\end{align*}
which reduces to
\begin{align*}
 \|  \Lambda^{-\frac{3}{4}+3\epsilon} (uv) \|_{{\mathcal L}^1_t {\mathcal L}^{4n}_x}  \lesssim \| u\|_{H^{s+\half-2\epsilon,\half+\epsilon}}  \| v \|_{H^{s-\half-2\epsilon,\half+\epsilon}} \, .
\end{align*}
This follows by Prop. \ref{TheoremE} with parameters $q=2$ , $r=8n$ , $\sigma= \frac{3}{4}-3\epsilon$ . This requires $s_1,s_2 < \frac{n}{2}-\frac{5}{8}$ and $s_1+s_2 \ge n-2+3\epsilon$ . We choose $s_1= \frac{n}{2}-\frac{5}{8}-\epsilon < s+\half-2\epsilon$ , $s_2=\frac{n}{2}-\frac{11}{8}+4\epsilon < s- \half-2\epsilon$ . \\
2. The estimate for the second term reduces to
\begin{align*}
 \|  \Lambda^{-\frac{3}{4}+2\epsilon} \Lambda_-^{-\half-\epsilon} (uv) \|_{{\mathcal L}^1_t {\mathcal L}^{4n}_x} 
\lesssim  \| \Lambda^{\frac{1}{2}-4\epsilon} \Lambda_-^{-\half+2\epsilon} u\|_{F^s} \| v \|_{H^{s-\half-4\epsilon,\half+\epsilon}} \, .,
\end{align*}
where we used Lemma \ref{Lemma8.10}. By Prop. \ref{Prop.4.8} we obtain the following bound for the left hand side:
$ \| \Lambda^{\frac{n}{2}-\frac{3}{2}+2\epsilon} (uv) \|_{{\mathcal L}^1_t {\mathcal L}^2_x} \, . $ Using the fractional Leibniz rule we estimate: \\
2.1.
\begin{align*}
\| u \Lambda^{\frac{n}{2}-\frac{3}{2}+2\epsilon} v \|_{{\mathcal L}^1_t {\mathcal L}^2_x} & \lesssim \|u\|_{{\mathcal L}^1_t {\mathcal L}^{\infty}_x} \| \Lambda^{\frac{n}{2}+2\epsilon} v \|_{L^{\infty}_t L^2_x}\\
& \lesssim \|\Lambda^{\frac{1}{2}-4\epsilon} \Lambda_-^{-\half+2\epsilon} u\|_{F^s}
\|v\|_{H^{s-\half-4\epsilon,\half+\epsilon}} \,.
\end{align*}
2.2. 
\begin{align*}
\| (\Lambda^{\frac{n}{2}-\frac{3}{2}+2\epsilon}\ u) v \|_{L^1_t  L^2_x} & \lesssim \| \Lambda^{\frac{n}{2}-\frac{3}{2}+2\epsilon} u\|_{L^2_t L^p_x} \|v\|_{L^2_t L^r_x} \\
& \lesssim \| u\|_{H^{s+\half-4\epsilon,0}} \|v\|_{L^2_t L^r_x} \, ,
\end{align*}
where $\frac{1}{p}=\half - \frac{\frac{9}{8}-4\epsilon}{n}$ and $\frac{1}{r} =\frac{\frac{9}{8}-4\epsilon}{n}$ , so that by Sobolev  $H^{s+\half-4\epsilon} \hookrightarrow H^{\frac{n}{2}-\frac{3}{2}+2\epsilon,p}$ for $s>\frac{n}{2}-\frac{7}{8}$ . For the last factor we may apply Prop. \ref{Str} directly in the case $n \ge 5$ , which gives 
$$ \|v\|_{L^2_t  L^r_x} \lesssim \|v\|_{H^{\frac{n}{2}-\frac{9}{8}-\half+4\epsilon,\half+\epsilon}} \lesssim \|v\|_{H^{s-\half-4\epsilon,\half+\epsilon}} \, , $$
whereas in the case $n=4$ we obtain by Prop. \ref{mStr} :
$$\|v\|_{L^2_t L^r_x} \lesssim \|v\|_{H^{\frac{n+1}{2}(\half-\frac{1}{r}),\half+\epsilon}} \lesssim \|v\|_{H^{s-\half-4\epsilon,\half+\epsilon}} \, , $$
as one easily checks.\\
3. The last term on the right hand side of (\ref{42'}) is reduced to
\begin{align*}
 \|  \Lambda^{-\frac{3}{4}+2\epsilon} \Lambda_-^{-\half-\epsilon} (uv) \|_{{\mathcal L}^1_t {\mathcal L}^{4n}_x} 
\lesssim  \|u\|_{H^{s+\half-4\epsilon,\half+\epsilon}} \| v \|_{H^{s-\half-4\epsilon,0}} \, ,
\end{align*}
As before we obtain the following bound for the left hand side:
$ \| \Lambda^{\frac{n}{2}-\frac{3}{2}+2\epsilon} (uv) \|_{{\mathcal L}^1_t {\mathcal L}^2_x} \, . $ Using the fractional Leibniz rule we estimate: \\
3.1. By Sobolev and Prop. \ref{Str} we obtain:
\begin{align*}
\| (\Lambda^{\frac{n}{2}-\frac{3}{2}+2\epsilon}) u) v \|_{L^1_t  L^2_x} & \lesssim \| \Lambda^{\frac{n}{2}-\frac{3}{2}+2\epsilon} u\|_{L^2_t L^r_x} \|v\|_{L^2_t L^p_x} \\
& \lesssim \| u\|_{H^{s+\half-4\epsilon,\half+\epsilon}} \|v\|_{H^{s-\half-4\epsilon,0}} \, ,
\end{align*}
where $\frac{1}{r}=\half - \frac{11}{8n}$ and $\frac{1}{p} =\frac{11}{8n}$ , so that by Sobolev  $H^{s-\half-4\epsilon} \hookrightarrow L^p$ for $s>\frac{n}{2}-\frac{7}{8}$.  For the first factor we may apply Prop. \ref{Str}, which gives 
$$ \|\Lambda^{\frac{n}{2}-\frac{3}{2}+2\epsilon}u\|_{L^2_t  L^r_x} \lesssim \|u\|_{H^{\frac{n}{2}-\frac{5}{8}+2\epsilon,\half+\epsilon}} \lesssim \|u\|_{H^{s+\half-4\epsilon,\half+\epsilon}} \, , $$
as one easily checks. \\
3.2. Similarly
\begin{align*}
\| u  (\Lambda^{\frac{n}{2}-\frac{3}{2}+2\epsilon} v) \|_{L^1_t  L^2_x} & \lesssim \| u\|_{L^2_t L^r_x} \|\Lambda^{\frac{n}{2}-\frac{3}{2}+2\epsilon} v\|_{L^2_t L^q_x} \\
& \lesssim \| u\|_{H^{s+\half-4\epsilon,\half+\epsilon}} \|v\|_{H^{s-\half-4\epsilon,0}} \, ,
\end{align*}
where $\frac{1}{q}=\half - \frac{5}{8n}$ and $\frac{1}{r} =\frac{5}{8n}$ , so that by Sobolev  $H^{s-\half-4\epsilon} \hookrightarrow H^{\frac{n}{2}-\frac{3}{2}+2\epsilon,q}$ for $s>\frac{n}{2}-\frac{7}{8}$.  For the first factor we may apply Prop. \ref{Str}, which gives 
$$ \|u\|_{L^2_t  L^r_x} \lesssim \|u\|_{H^{\frac{n}{2}-\frac{5}{8}-\half,\half+\epsilon}} \lesssim \|v\|_{H^{s+\half-4\epsilon,\half+\epsilon}} \, . $$
\end{proof}

\begin{proof}[Proof of (\ref{26})]
We may reduce to
$$ \|\Lambda_+^{-1} \Lambda_-^{\epsilon -1} \Lambda_+ Q(u,v)\|_{H^{r,\half+\epsilon}} \lesssim \|\Lambda u \|_{F^s} \|\Lambda_+ v\|_{H^{r-1,\half+\epsilon}} \, . $$ 
Now we use (\ref{42''}) with $\epsilon=0$ for $Q(u,v)$ and estimate the six terms as follows: \\
1. The estimate for the first term is reduced to (using the trivial estimate $\Lambda_-^{2\epsilon} u \precsim  \Lambda_+^{2\epsilon } u $) :
$$ \|uv\|_{H^{r-1+2\epsilon,0}} \lesssim \|u\|_{H^{s,\half+\epsilon}} \| v \|_{H^{r-\half,\half+\epsilon}} \, , $$
which follows from Prop. \ref{Prop.1.2}. \\
2. The estimate for the second term reduces to
$$ \|uv\|_{H^{r-1,2\epsilon}} \lesssim \|u\|_{H^{s+\half,\half+\epsilon}} \|v\|_{H^{r-1,\half+\epsilon}} \, , $$
which is a consequence of Cor. \ref{Cor.1.2} under our assumption $2s-r > \frac{n}{2}$ . \\
3. The third term requires
$$ \|uv\|_{H^{r-1,-\half+2\epsilon}} \lesssim \|u\|_{H^{s,0}} \|v\|_{H^{r-\half,\half+\epsilon}} \, . $$ 
4. The forth term similarly reduces to
$$ \|uv\|_{H^{r-1,-\half+2\epsilon}} \lesssim \|u\|_{H^{s+\half,\half+\epsilon}} \|v\|_{H^{r-1,0}} \, . $$ 
6. The estimate for the sixth term similarly follows from
$$ \|uv\|_{H^{r-1,-\half+2\epsilon}} \lesssim \|u\|_{H^{s,\half+\epsilon}} \|v\|_{H^{r-\half,0}} \, . $$ 
These three estimates follow from Cor. \ref{Cor.1.1}. \\
5. The fifth term is the most complicated one. It follows from
$$
\|uv\|_{H^{r-1,-\half+2\epsilon}} \lesssim ( \|u\|_{H^{s+\half,0}} + \|\Lambda^{\frac{1}{4}+5\epsilon} u \|_{{\mathcal L}^1_t {\mathcal L}^{4n}_x} ) \|v\|_{H^{r-1,\half+\epsilon}} \, . 
$$
By use of Lemma \ref{Lemma8.10} we obtain
$$
 \|u v\|_{H^{r-1,-\half+2\epsilon}}   
\lesssim \| (\Lambda^{3\epsilon} u) v\|_{H^{r-1,-\half-\epsilon}} 
+  \| (\Lambda^{-\half+2\epsilon} u) v\|_{H^{r-1,0}}\, . 
$$
The last term is easily estimated by Sobolev for $s > \frac{n}{2}-1$ :
$$  \| (\Lambda^{-\half+2\epsilon} u) v\|_{H^{r-1,0}}  \lesssim \|u\|_{H^{s+1-2\epsilon,0}} \|v\|_{H^{r-1,\half+2\epsilon}} \, . $$
For the first term we want to show
\begin{align}
\label{69}
 \|uv\|_{H^{r-1,-\half-\epsilon}} & \lesssim (\| \Lambda^{\frac{1}{4}+2\epsilon} u \|_{{\mathcal L}^1_t {\mathcal L}^{4n}_x} + \|u\|_{H^{s+\half-3\epsilon,0}} ) \|v\|_{H^{r-1,\half+\epsilon}}
\end{align}
Consider first the case $ r \geq \frac{3}{2}$.
 (\ref{69}) is handled by the fractional Leibniz rule as follows:\\
1.
\begin{align*}
\| u \Lambda^{r-1} v \|_{H^{0,-\half- \epsilon}} &\lesssim \|u \Lambda^{r-1} v\|_{{\mathcal L}^1_t {\mathcal L}^2_x} \lesssim \|u\|_{{\mathcal L}^1_t {\mathcal L}^{\infty}_x} \| \Lambda^{r-1} v\|_{L^{\infty}_t L^2_x} \\
& \lesssim \| \Lambda^{\frac{1}{4}+2\epsilon} u\|_{{\mathcal L}^1_t {\mathcal L}^{4n}_x} 
 \|v\|_{H^{r-1,\half+\epsilon}} \, , 
\end{align*}
where we used Prop. \ref{Cor.4.6}  and Prop. \ref{Prop.4.3} .\\
2. We obtain the estimate 
\begin{align}
\label{100}
 \| (\Lambda^{r-1}u)v\|_{H^{0,-\half-\epsilon}} \lesssim \|u\|_{H^{s+\half-3\epsilon,0}} \|v\|_{H^{r-1,\half+\epsilon}} \, , 
\end{align}
by Prop. \ref{Prop.1.2} under our assumptions $s > \frac{n}{2} - 1$ , $r \geq \frac{3}{2}$ and $s> r-\frac{3}{2}$ , in which case one easily checks the necessary conditions. Thus under these assumptions we proved (\ref{69}). \\
Next we consider the case $ r \le \half $ . (\ref{69}) follows by duality from the estimate
$$
 \|uv\|_{H^{1-r,-\half-\epsilon}} \lesssim (\| \Lambda^{\frac{1}{4}+2\epsilon} u \|_{{\mathcal L}^1_t {\mathcal L}^{4n}_x} + \|u\|_{H^{s+\half-3\epsilon,0}} ) \|v\|_{H^{1-r,\half+\epsilon}} \, . $$
Using the fractional Leibniz rule we obtain
$$ \| (\Lambda^{1-r} u) v\|_{H^{0,-\half-\epsilon}} \lesssim \|u\|_{H^{s+\half-3\epsilon,0}} \|v\|_{H^{1-r,\half+\epsilon}} $$
by Prop. \ref{Prop.1.2} under the assumptions $s>\frac{n}{2}-1$ and $ r \le \half$ (and $s > \half - r$). Moreover similarly as (\ref{100}) : 
\begin{align*}
\|u \Lambda^{1-r} v\|_{H^{0,-\half-\epsilon}}
 \lesssim \| \Lambda^{\frac{1}{4}+2\epsilon} u\|_{{\mathcal L}^1_t {\mathcal L}^{4n}_x} \|v\|_{H^{1-r,\half+\epsilon}} \, .
\end{align*}
Thus we have proven (\ref{69}) for $r \le \half$ and $s> \frac{n}{2} - 1$ as well. \\
By interpolation we also obtain this estimate in the remaining case $ \half < r < \frac{3}{2} $ . 

\end{proof}

\begin{proof}[Proof of (\ref{30}) and (\ref{31})]
(\ref{30}) reduces to the following estimates
$$
\|uv\|_{H^{s-1,-\half+\epsilon}}  \lesssim \|u\|_{H^{s,\half+\epsilon}} \|v\|_{H^{s+1,\half+\epsilon}} \, ,$$
which easily follows from  Prop. \ref{SML}, because $(1-s)+s+(s+1) = s+2 > \frac{n}{2} +1$, and
\begin{align*}
\| \Lambda^{-\frac{5}{4}+5\epsilon} \Lambda_-^{-\half+\epsilon} (uv) \|_{{\mathcal L}^1_t {\mathcal L}^{4n}_x} &\lesssim \| \Lambda^{-\frac{5}{4}+7\epsilon} \Lambda_-^{-\half-\epsilon} (uv) \|_{{\mathcal L}^1_t {\mathcal L}^{4n}_x} \\
&\lesssim \| \Lambda^{\frac{n}{2}-2+7\epsilon} (uv) \|_{{\mathcal L}^1_t {\mathcal L}^2_x} \lesssim \|u\|_{H^{s,0}} \|v\|_{H^{s+1,0}} \, , 
\end{align*}
where we used Prop. \ref{Prop.4.8} and Prop. \ref{SML}, because $(2-\frac{n}{2}-7\epsilon)+s+(s+1) > \frac{n}{2}+1$. The estimate (\ref{31}) is proven in exactly the same way.
\end{proof}

\begin{proof}[Proof of (\ref{33}) and (\ref{34})]
(\ref{33}) reduces to the following estimate
$$ \|uv\|_{H^{r-1,-\half+\epsilon}} \lesssim \|u\|_{H^{s+2,\half+\epsilon}} \|v\|_{H^{r-1,\half+\epsilon}} \, , $$
which follows from Prop. \ref{SML}, because $s+2 > \frac{n}{2} +1$ . Similarly (\ref{34}) can be proven.
\end{proof}

\begin{proof}[Proof of (\ref{35})]
{\bf A.} The estimate for the first part of the $F^s$-norm reduces to
$$ \|uvw\|_{H^{s-1,-\half+2\epsilon}} \lesssim \|u\|_{H^{r+1,\half+\epsilon}} \|v\|_{H^{s,\half+\epsilon}}  \|w\|_{H^{s,\half+\epsilon}} \, . $$
We first assume $s \le \frac{n}{2}-\frac{1}{8}$ .
In a first step we apply Cor. \ref{Cor.1.1}, using our assumptions $s> \frac{n}{2}-\frac{7}{8}$ and $r > \frac{n}{2}-\frac{7}{4} $ . By elementary calculations we obtain
$$ \|uvw\|_{H^{s-1,-\half+2\epsilon}} \lesssim \|u\|_{H^{r+1,\half+\epsilon}} \|vw\|_{H^{2s-\frac{n}{2}+\frac{1}{8},0}} \, .$$
For the second step we obtain by Prop. \ref{Prop.1.2} 
$$\|vw\|_{H^{2s-\frac{n}{2}+\frac{1}{8},0}} \lesssim \|v\|_{H^{s,\half+\epsilon}} \|w\|_{H^{s,\half+\epsilon}} \, . $$
If however $s > \frac{n}{2}-\frac{1}{8}$ we obtain similarly by easy calculations 
$$ \|uvw\|_{H^{s-1,-\half+2\epsilon}} \lesssim \|u\|_{H^{r+1,\half+\epsilon}} \|vw\|_{H^{s,0}} 
\lesssim \|u\|_{H^{r+1,\half+\epsilon}} \|v\|_{H^{s,\half+\epsilon}} \|w\|_{H^{s,\half+\epsilon}}
\, .$$ 
{\bf B.} The estimate for the second part of the $F^s$-norm reduces to
$$ \|\Lambda^{-\frac{5}{4}+5\epsilon} \Lambda_-^{-\half+\epsilon}(uvw) \|_{{\mathcal L}^1_t {\mathcal L}^{4n}_x} \lesssim \|u\|_{H^{r+1,\half+\epsilon}} \|v\|_{H^{s,\half+\epsilon}}  \|w\|_{H^{s,\half+\epsilon}} \, . $$
By Prop. \ref{Prop.4.8} we obtain
$$\|\Lambda^{-\frac{5}{4}+5\epsilon} \Lambda_-^{-\half+\epsilon}(uvw) \|_{{\mathcal L}^1_t {\mathcal L}^{4n}_x} \lesssim \|\Lambda^{\frac{n}{2}-2+7\epsilon} \Lambda_-^{-\half+\epsilon}(uvw) \|_{{\mathcal L}^1_t {\mathcal L}^2_x} \, , $$
which by the fractional Leibniz rule and symmetry in $v$ and $w$ is estimated as follows: \\
1. By Sobolev and Prop. \ref{Str} we obtain under our assumption $s>\frac{n}{2}-\frac{7}{8}$ and $r > \frac{n}{2}-\frac{7}{4}$ :
\begin{align*}
\|(\Lambda^{\frac{n}{2}-2+7\epsilon} u)vw \|_{L^1_t L^2_x} & \lesssim 
\|\Lambda^{\frac{n}{2}-2+7\epsilon} u\|_{L^{\infty}_t L^{\frac{4n}{2n-5}}_x} \| v \|_{L^2_t L^{\frac{8n}{5}}_x} \| w \|_{L^2_t L^{\frac{8n}{5}}_x} \\
& \lesssim  \|u\|_{H^{r+1,\half+\epsilon}} \|v\|_{H^{s,\half+\epsilon}}  \|w\|_{H^{s,\half+\epsilon}} \, . 
\end{align*}
2. Similarly we also obtain
\begin{align*}
\| u (\Lambda^{\frac{n}{2}-2+7\epsilon}v) w \|_{L^1_t L^2_x} & \lesssim 
\| u \|_{L^2_t L^{2n}_x} \|\Lambda^{\frac{n}{2}-2+7\epsilon} v\|_{L^{\infty}_t L^{\frac{2n}{n-2}}_x} \| w \|_{L^2_t L^{2n}_x} \\
& \lesssim  \|u\|_{H^{r+1,\half+\epsilon}} \|v\|_{H^{s,\half+\epsilon}}  \|w\|_{H^{s,\half+\epsilon}} \, . 
\end{align*}
\end{proof}

\begin{proof}[Proof of (\ref{36})]
{\bf A.} We may reduce to
$$ \| u \Lambda^{-1} (vw)\|_{H^{s-1,-\half+2\epsilon}} \lesssim \|u\|_{H^{r,\half+\epsilon}} \|v\|_{H^{s,\half+\epsilon}} \|w\|_{H^{s,\half+\epsilon}} \, . $$
In the case $\frac{n}{2}-\frac{7}{8} < s \le \frac{n}{2} - \half$ we use Cor. \ref{Cor.1.1} and Prop. \ref{Prop.1.2} and obtain
\begin{align*}
 \| u \Lambda^{-1} (vw)\|_{H^{s-1,-\half+2\epsilon}} & \lesssim \|u\|_{H^{r,\half+\epsilon}} \|\Lambda^{-1} (vw)\|_{H^{-\frac{n}{2}+2s+\frac{3}{2},0}} \\
& \lesssim  \|u\|_{H^{r,\half+\epsilon}} \|v\|_{H^{s,\half+\epsilon}} \|w\|_{H^{s,\half+\epsilon}} \, .
\end{align*}
In the case $s > \frac{n}{2} - \half$ we obtain similarly
\begin{align*}
 \| u \Lambda^{-1} (vw)\|_{H^{s-1,-\half+2\epsilon}} & \lesssim \|u\|_{H^{r,\half+\epsilon}} \|\Lambda^{-1} (vw)\|_{H^{s+1,0}} \\
& \lesssim  \|u\|_{H^{r,\half+\epsilon}} \|v\|_{H^{s,\half+\epsilon}} \|w\|_{H^{s,\half+\epsilon}} \, .
\end{align*}
{\bf B.} We have to control
$$\|\Lambda^{-\frac{5}{4}+5\epsilon} \Lambda_-^{-\half+\epsilon}(u\Lambda^{-1}(vw)) \|_{{\mathcal L}^1_t {\mathcal L}^{4n}_x} \lesssim \|\Lambda^{\frac{n}{2}-2+7\epsilon} (u\Lambda^{-1}(vw)) \|_{{\mathcal L}^1_t {\mathcal L}^2_x} \, , $$
where we used Prop. \ref{Prop.4.8}. By the fractional Leibniz rule we have to consider two terms. \\
1. By Sobolev and Prop. \ref{Str} we obtain
\begin{align*}
&\|(\Lambda^{\frac{n}{2}-2+7\epsilon} u)\Lambda^{-1}(vw) \|_{{\mathcal L}^1_t {\mathcal L}^2_x} \lesssim \| \Lambda^{\frac{n}{2}-2+7\epsilon} u \|_{L^{\infty}_t L^2_x} \| \Lambda^{-1}(vw)\|_{L^1_t L^{\infty}_x} \\
&\lesssim  \|u\|_{H^{r,\half+\epsilon}} \|v\|_{L^2_t L^{\frac{8n}{3}}_x} \|w\|_{L^2_t L^{\frac{8n}{3}}_x} 
\lesssim
\| u \|_{H^{r,\half+\epsilon}} \|v\|_{H^{s,\half+\epsilon}}  \|w\|_{H^{s,\half+\epsilon}} \, .
\end{align*}
2. We have to estimate  $\| u \Lambda^{\frac{n}{2}-3+7\epsilon}(vw) \|_{L^1_t L^2_x} $  , which in the case $n \ge 6$ by symmetry and the fractional Leibniz rule reduces to $\| u (\Lambda^{\frac{n}{2}-3+7\epsilon}v)w \|_{L^1_t L^2_x} $ . By Sobolev and Prop. \ref{Str} we obtain
\begin{align*}
 \| u (\Lambda^{\frac{n}{2}-3+7\epsilon}v)w \|_{L^1_t L^2_x}
 & \lesssim \|u\|_{L^2_t L^{\frac{4n}{5}}_x} \| \Lambda^{\frac{n}{2}-3+7\epsilon} v \|_{L^{\infty}_t L^{\frac{8n}{4n-17}}_x} \|w\|_{L^2_t L^{\frac{8n}{7}}_x} \\
& \lesssim \|u\|_{H^{r,\half+\epsilon}} \|v\|_{H^{s,\half+\epsilon}} \|w\|_{H^{s,\half+\epsilon}} \, . 
\end{align*}
If $4 \le n \le 5$ we obtain by the same means
\begin{align*}
\| u \Lambda^{\frac{n}{2}-3+7\epsilon}(vw) \|_{L^1_t L^2_x} & \lesssim \|u\|_{L^{\infty}_t L^{\frac{4n}{7}}_x} \| \Lambda^{\frac{n}{2}-3+7\epsilon}(vw)\|_{L^1_t L^{\frac{4n}{2n-7}}_x} \\
& \lesssim \|u\|_{H^{r,\half+\epsilon}} \|vw\|_{L^1_t L^{\frac{4n}{5}}_x} \\
&\lesssim \|u\|_{H^{r,\half+\epsilon}} \|v\|_{L^2_t L^{\frac{8n}{5}}_x} \|w\|_{L^2_t L^{\frac{8n}{5}}_x} \\
& \lesssim
 \|u\|_{H^{r,\half+\epsilon}} \|v\|_{H^{s,\half+\epsilon}} \|w\|_{H^{s,\half+\epsilon}} \, .
\end{align*}
\end{proof}

\begin{proof}[Proof of (\ref{38})]
{\bf A.} We have to show the estimate
$$ \|uvw\|_{H^{s-1,-\half+2\epsilon}} \lesssim \|u\|_{H^{s,\half+\epsilon}} \|v\|_{H^{s,\half+\epsilon}} \|w\|_{H^{s,\half+\epsilon}} \, . $$
It is sufficient to consider the (minimal) value $s= \frac{n}{2}-1 + \epsilon$ , which immediately implies this for any larger $s$. This follows from Cor. \ref{Cor.1.1} and Prop. \ref{Prop.1.2} :
\begin{align*}
 \|uvw\|_{H^{\frac{n}{2}-2+\epsilon,-\half+2\epsilon}} & \lesssim \|u\|_{H^{\frac{n}{2}-1+\epsilon,\half+\epsilon}} \|vw\|_{H^{\frac{n}{2}-\frac{3}{2}+\epsilon,0}} \\
&\lesssim \|u\|_{H^{\frac{n}{2}-1+\epsilon,\half+\epsilon}}  \|v\|_{H^{\frac{n}{2}-1+\epsilon,\half+\epsilon}}  \|w\|_{H^{\frac{n}{2}-1+\epsilon,\half+\epsilon}} \, . 
\end{align*}
{\bf B.} For the second part of the $F^s$-norm we use Prop. \ref{Prop.4.8} and obtain
\begin{align*}
\| \Lambda^{-\frac{5}{4}+5\epsilon} \Lambda_-^{-\half+\epsilon} (uvw)\|_{{\mathcal L}^1_t {\mathcal L}^{4n}_x} &\lesssim \|\Lambda^{\frac{n}{2}-2+7\epsilon} (uvw)) \|_{{\mathcal L}^1_t {\mathcal L}^2_x} \, .
\end{align*}
Now by symmetry we only have to estimate
\begin{align*}
\| (\Lambda^{\frac{n}{2}-2+7\epsilon} u) vw\|_{L^1_t L^2_x} & \lesssim \| \Lambda^{\frac{n}{2}-2+7\epsilon} u\|_{L^{\infty}_t L^{\frac{2n}{n-2}}_x} \| v\|_{L^2_t L^{2n}_x} \| w\|_{L^2_t L^{2n}_x} \\
& \lesssim \|u\|_{H^{s,\half+\epsilon}} \|v\|_{H^{s,\half+\epsilon}} \|w\|_{H^{s,\half+\epsilon}} \, .
\end{align*}
\end{proof}

\begin{proof}[Proof of (\ref{39})]
It suffices to show
$$ \|uvw\|_{H^{r-1,-\half+2\epsilon}} \lesssim \|u\|_{H^{s,\half+\epsilon}} \|v\|_{H^{s,\half+\epsilon}} \|w\|_{H^{r,\half+\epsilon}} \, , $$
which follows by Cor. \ref{Cor.1.1} and Prop. \ref{Prop.1.2}  in the case $r \le \frac{n}{2}-\half$ from
$$ \|uvw\|_{H^{r-1,-\half+2\epsilon}} \lesssim \|uv\|_{H^{\frac{n}{2}-\frac{3}{2}+\epsilon,0}}\|w\|_{H^{r,\half+\epsilon}} \lesssim \|u\|_{H^{s,\half+\epsilon}} \|v\|_{H^{s,\half+\epsilon}} \|w\|_{H^{r,\half+\epsilon}} \, , $$
whereas in the case $ r > \frac{n}{2} - \half$ we obtain 
$$ \|uvw\|_{H^{r-1,-\half+2\epsilon}} \lesssim \|uv\|_{H^{r-1,0}}\|w\|_{H^{r,\half+\epsilon}} \lesssim \|u\|_{H^{s,\half+\epsilon}} \|v\|_{H^{s,\half+\epsilon}} \|w\|_{H^{r,\half+\epsilon}} \, , $$
using our assumption $2s-r > \frac{n}{2}$ .
\end{proof}

\begin{proof}[Proof of (\ref{37})]
{\bf A.} For the first part of the $F^s$-norm we have to show
$$ \| \Lambda^{-1} (uv) wz \|_{H^{s-1,-\half+2\epsilon}} \lesssim \|u\|_{H^{s,\half+\epsilon}} 
\|v\|_{H^{s,\half+\epsilon}} \|w\|_{H^{s,\half+\epsilon}} \|z\|_{H^{s,\half+\epsilon}} \, . $$
It suffices to consider the minimal value $s=\frac{n}{2}-\frac{7}{8}+\epsilon$ , which by Cor. \ref{Cor.1.1} and Prop. \ref{Prop.3.8} can be estimated as follows:
\begin{align*}
&\| \Lambda^{-1} (uv) wz \|_{H^{\frac{n}{2}-\frac{15}{8}+\epsilon,-\half+2\epsilon}}  \lesssim \| \Lambda^{-1} (uv)  \|_{H^{\frac{n}{2}-\frac{3}{4}+2\epsilon,\half+\epsilon}} \| wz  \|_{H^{\frac{n}{2}-\frac{13}{8}+2\epsilon,0}} \\
& \hspace{1em} \lesssim \| uv  \|_{H^{\frac{n}{2}-\frac{7}{4}+2\epsilon,\half+\epsilon}} \| wz  \|_{H^{\frac{n}{2}-\frac{13}{8}+2\epsilon,0}} \\
& \hspace{1em}\lesssim \| u \|_{H^{\frac{n}{2}-\frac{7}{8}+2\epsilon,\half+\epsilon}} \|v  \|_{H^{\frac{n}{2}-\frac{7}{8}+2\epsilon,\half+\epsilon}} \| w\|_{H^{\frac{n}{2}-\frac{7}{8}+2\epsilon,\half+\epsilon}}  \|z  \|_{H^{\frac{n}{2}-\frac{7}{8}+2\epsilon,\half+\epsilon}} \, . 
\end{align*}
{\bf B.} By Prop. \ref{Prop.4.8} and the fractional Leibniz rule we obtain
\begin{align*}
&\| \Lambda^{-\frac{5}{4}+5\epsilon} \Lambda_-^{-\half+\epsilon} (\Lambda^{-1}(uv)wz))\|_{{\mathcal L}^1_t {\mathcal L}^{4n}_x} \lesssim \| \Lambda^{\frac{n}{2}-2+7\epsilon} ((\Lambda^{-1}(uv)wz)))\|_{{\mathcal L}^1_t {\mathcal L}^2_x}\\
& \lesssim \| \Lambda^{\frac{n}{2}-3+7\epsilon}(uv) (wz)\|_{L^1_t L^2_x} + \| \Lambda^{-1}(uv)\Lambda^{\frac{n}{2}-2+7\epsilon} (wz)\|_{L^1_t L^2_x} \, .
\end{align*}
1. In the case $ n \ge 6$ we reduce the estimate for the first term by symmetry to
\begin{align*}
\| (\Lambda^{\frac{n}{2}-3+7\epsilon}u)v wz\|_{L^1_t L^2_x} & \lesssim \| \Lambda^{\frac{n}{2}-3+7\epsilon}u\|_{L^{\infty}_t L^{\frac{2n}{n-4}}_x} \|v\|_{L^{\infty}_t L^n_x} \|w\|_{L^2_t L^{2n}_x} \|z\|_{L^2_t L^{2n}_x} \\
& \lesssim \|u\|_{H^{\frac{n}{2}-1+7\epsilon,\half+\epsilon}} \|v\|_{H^{\frac{n}{2}-1,\half+\epsilon}} \|w\|_{H^{\frac{n}{2}-1,\half+\epsilon}} \|z\|_{H^{\frac{n}{2}-1,\half+\epsilon}} 
\end{align*}
by Sobolev and Prop. \ref{Prop.1.2}, whereas for $4 \le n \le 5$ we similarly obtain
\begin{align*}
&\| \Lambda^{\frac{n}{2}-3+7\epsilon}(uv) (wz)\|_{L^1_t L^2_x}  \lesssim \| \Lambda^{\frac{n}{2}-3+7\epsilon}(uv)\|_{L^{\infty}_t L^{\frac{2n}{n-2}}_x} \|w\|_{L^2_t L^{2n}_x} \|z\|_{L^2_t L^{2n}_x} \\
& \hspace{1em} 
\lesssim 
\| \Lambda^{\frac{n}{2}-2+7\epsilon}(uv)\|_{L^{\infty}_t L^2_x}  \|w\|_{H^{\frac{n}{2}-1,\half+\epsilon}} \|z\|_{H^{\frac{n}{2}-1,\half+\epsilon}} \\
& \hspace{1em} \lesssim 
\|u\|_{L^{\infty}_t H^{\frac{n}{2}-1+4\epsilon}_x}  \|v\|_{L^{\infty}_t H^{\frac{n}{2}-1+4\epsilon}_x} \|w\|_{H^{\frac{n}{2}-1,\half+\epsilon}} \|z\|_{H^{\frac{n}{2}-1,\half+\epsilon}} \\
& \hspace{1em} \lesssim \|u\|_{H^{\frac{n}{2}-1+4\epsilon,\half+\epsilon}} \|v\|_{H^{\frac{n}{2}-1+4\epsilon,\half+\epsilon}} \|w\|_{H^{\frac{n}{2}-1,\half+\epsilon}} \|z\|_{H^{\frac{n}{2}-1,\half+\epsilon}} 
\end{align*}
by Prop. \ref{SML} and Prop. \ref{Prop.1.2}. \\
2. The estimate for the second term is as follows
\begin{align*}
 &\| \Lambda^{-1}(uv)\Lambda^{\frac{n}{2}-2+7\epsilon} (wz)\|_{L^1_t L^2_x} \\
& \lesssim \| \Lambda^{-1}(uv)\|_{L^1_t L^{\infty}_x} \|\Lambda^{\frac{n}{2}-2+7\epsilon} (wz)\|_{L^{\infty}_t L^2_x} \\
&  \lesssim \|uv\|_{L^1_t L^{n+\epsilon}_x} (\|\Lambda^{\frac{n}{2}-2+2\epsilon} w\|_{L^{\infty}_t L^{\frac{2n}{n-2}}_x} \| z\|_{L^{\infty}_t L^n_x} + \| w\|_{L^{\infty}_t L^n_x} \|\Lambda^{\frac{n}{2}-2+7\epsilon} z\|_{L^{\infty}_t L^{\frac{2n}{n-2}}_x}) \\
& \lesssim \|u\|_{L^2_t L^{2(n+\epsilon)}_x}  \|v\|_{L^2_t L^{2(n+\epsilon)}_x} \|w\|_{H^{\frac{n}{2}-1+2\epsilon,\half+\epsilon}}  \|z\|_{H^{\frac{n}{2}-1+2\epsilon,\half+\epsilon}} \\
& \lesssim \|u\|_{H^{\frac{n}{2}-1+O(\epsilon),\half+\epsilon}}  \|v\|_{H^{\frac{n}{2}-1+O(\epsilon),\half+\epsilon}}  \|w\|_{H^{\frac{n}{2}-1+7\epsilon,\half+\epsilon}}  \|z\|_{H^{\frac{n}{2}-1+7\epsilon,\half+\epsilon}} \, ,
\end{align*}
where we used Sobolev and Strichartz as before.
\end{proof}

\begin{proof}[Proof of (\ref{40})]
We have to show
$$ \|uv wz \|_{H^{r-1,-\half+2\epsilon}} \lesssim \|u\|_{H^{s,\half+\epsilon}} 
\|v\|_{H^{s,\half+\epsilon}} \|w\|_{H^{s,\half+\epsilon}} \|z\|_{H^{s,\half+\epsilon}} \, . $$
By our assumption $2s-r > \frac{n}{2}$ the left hand side is bounded by the term \\
$\|uv wz \|_{H^{2s-\frac{n}{2}-1-\epsilon,-\half+2\epsilon}}$ . It suffices to prove the remaining estimate for the (minimal) value $s= \frac{n}{2}-\frac{7}{8}$ . By Cor. \ref{Cor.1.1} and Prop. \ref{SML} we obtain
\begin{align*}
\|uv wz \|_{H^{\frac{n}{2}-\frac{11}{4}-\epsilon,-\half+2\epsilon}} & \lesssim \|uv\|_{H^{\frac{n}{2}-\frac{15}{8}+O(\epsilon),\half+\epsilon}} \| wz \|_{H^{\frac{n}{2}-\frac{11}{8}+O(\epsilon),0}} \\
& \lesssim \|u\|_{H^{\frac{n}{2}-\frac{7}{8},\half+\epsilon}}    \|v\|_{H^{\frac{n}{2}-\frac{7}{8},\half+\epsilon}} \| w\|_{H^{\frac{n}{2}-\frac{7}{8},\half+\epsilon}}  \|z \|_{H^{\frac{n}{2}-\frac{7}{8},\half+\epsilon}}
\end{align*}
\end{proof}

\section{ Proof of the nonlinear estimates in the case $n=3$}

\begin{proof}[Proof of (\ref{28})]
 We recall (\ref{41'}) for $\alpha = \epsilon$ :
\begin{align*}
Q_0(u,v) &\precsim  D_-^{1-\epsilon} (D_+ u D_+^{\epsilon} v) +   D_-^{1-\epsilon} (D_+^{\epsilon} u D_+ v)   \\
  &  \hspace{1em}  + (D_+ D_-^{1-\epsilon} u)(D_+^{\epsilon} v) + (D_+^{\epsilon} u)(D_+ D_-^{1-\epsilon} v) \, .
\end{align*}
Thus we have to show the following estimates and remark that we only have to consider the first and third term, because the last two terms are equivalent by symmetry. \\
1. For the first term it suffices to show
\begin{align*}
 &\| \Lambda_+^{-1} \Lambda_-^{-1+\epsilon} \Lambda^{r-1} \Lambda_+  D_-^{1-\epsilon}(D_+ u D_+^{\epsilon} v) \|_{H^{0,\half+\epsilon}} \\
&\hspace{1em}\lesssim \| \Lambda^{s-1} \Lambda_+ u \|_{H^{0,\half+\epsilon}} \| \Lambda^{s-1} \Lambda_+ v \|_{H^{0,\half+\epsilon}} \,. 
\end{align*}
This follows from
$$ \|uv\|_{H^{r-1,\half+\epsilon}} \lesssim \|u\|_{H^{s-1,\half+\epsilon}} \|v\|_{H^{s-\epsilon,\half+\epsilon}} \, , $$
which is a consequence of Prop. \ref{Prop.1.2'} under our conditions $s\ge r$ , $2s-r > \frac{3}{2}$ and $s-1+s-\epsilon > \half+\epsilon$ , thus here we need $s > \frac{3}{4}$ .  \\
2. For the second term we show
\begin{align*}
& \| \Lambda_+^{-1} \Lambda_-^{-1+\epsilon} \Lambda^{r-1} \Lambda_+ (D_+ D_-^{1-\epsilon}u D_+^{\epsilon} v) \|_{H^{0,\half+\epsilon}}
\lesssim \| \Lambda^{s-1} \Lambda_+ u \|_{H^{0,\half+\epsilon}} \| \Lambda^{s-1} \Lambda_+ v \|_{H^{0,\half+\epsilon}} \,.
\end{align*}
Thus it suffices to show
$$\|uv\|_{H^{r-1,-\half+2\epsilon}} \lesssim \|u\|_{H^{s-1,-\half+2\epsilon}} \|v\|_{H^{s-\epsilon,\half+\epsilon}} \, $$
which is a consequence of Prop. \ref{Prop.1.2'} as in 1. under the same assumptions. 
\end{proof}

\begin{proof}[Proof of (\ref{27})]
 We use (\ref{42'}).
Thus we have to show the following estimates and remark that we only have to consider the first two terms, because the last two terms are equivalent by symmetry. \\
1. For the first term it suffices to show
\begin{align*}
 &\| \Lambda_+^{-1} \Lambda_-^{-1+\epsilon} \Lambda^{r-1} \Lambda_+ \Lambda_+^{\half-2\epsilon} \Lambda_-^{\half-2\epsilon}(\Lambda_+^{\half+2\epsilon} u  \Lambda_+^{\half+2\epsilon} v) \|_{H^{0,\half+\epsilon}} \\
&\hspace{1em}\lesssim \| \Lambda^{s-1} \Lambda_+ u \|_{H^{0,\frac{3}{4}+\epsilon}} \| \Lambda^{s-1} \Lambda_+ v \|_{H^{0,\frac{3}{4}+\epsilon}} \,. 
\end{align*}
This follows from
$$ \|uv\|_{H^{r-\half-2\epsilon,0}} \lesssim \|u\|_{H^{s-\half-2\epsilon,\frac{3}{4}+\epsilon}} \|v\|_{H^{s-\half-2\epsilon,\frac{3}{4}+\epsilon}} \, , $$
which is a consequence of Prop. \ref{Prop.1.2'} with parameters $s_0=\half-r+2\epsilon$ , $s_1 = s_2 = s-\half-2\epsilon$ , $b_0 =0$ , $b_1=b_2=\frac{3}{4}+\epsilon$ , so that $s_0+s_1+s_2 > 1 $ , if $2s-r > \frac{3}{2}$ , and $s_0+s_1+s_2+s_1+s_2 > \frac{3}{2}$,  if $4s-r > \frac{11}{4}$ , which holds under our assumptions. \\
2. For the second term we show
\begin{align*}
& \| \Lambda_+^{-1} \Lambda_-^{-1+\epsilon} \Lambda^{r-1} \Lambda_+ \Lambda_+^{\half-2\epsilon} (\Lambda_+^{\half+2\epsilon} \Lambda_-^{\half-2\epsilon}u \Lambda_+^{\half+2\epsilon} v) \|_{H^{0,\half+\epsilon}} \\
&\lesssim \| \Lambda^{s-1} \Lambda_+ u \|_{H^{0,\frac{3}{4}+\epsilon}} \| \Lambda^{s-1} \Lambda_+ v \|_{H^{0,\frac{3}{4}+\epsilon}} \,.
\end{align*}
Using $\Lambda_-^{4\epsilon}u \precsim \Lambda_+^{4\epsilon}u$ it suffices to show
$$\|uv\|_{H^{r-\half+2\epsilon,-\half-2\epsilon}} \lesssim \|u\|_{H^{s-\half-2\epsilon,\frac{1}{4}}} \|v\|_{H^{s-\half-2\epsilon,\frac{3}{4}+\epsilon}} \, $$
which is a consequence of Prop. \ref{Prop.1.2'} as in 1., if $2s-r > \frac{3}{2}$ and $3s-2r > \frac{7}{4}$ , which holds under our assumptions.
\end{proof}

\begin{proof}[Proof of (\ref{32})]
As before it is easy to see that we can reduce to
$$ \|uv\|_{H^{s-1,-\frac{1}{4}+2\epsilon}} \lesssim \|u\|_{H^{r+1,\half+\epsilon}} \|v\|_{H^{r,\half+\epsilon}}\,. $$
This is a consequence of Prop. \ref{Prop.1.2'}. One easily checks that it can be applied  under the conditions $s \le r+1$ , $2r-s > -1$ , $4r-s>-\frac{7}{4}$ and $3r-2s>-2$ , all of which are satisfied under our assumptions.
\end{proof}

\begin{proof}[Proof of (\ref{24}) and (\ref{25})]
We have to prove
$$ \|\Lambda_+^{-1} \Lambda_-^{-1+\epsilon} Q(u,v)\|_{H^{s,\frac{3}{4}+\epsilon}} \lesssim \|\Lambda_+ u\|_{H^{s,\frac{3}{4}+\epsilon}} \|\Lambda_+ v\|_{H^{s,\frac{3}{4}+\epsilon}} \, . $$
By (\ref{42''}) we reduce to the following estimates:
\begin{align*}
\|uv\|_{H^{s-\half+2\epsilon,\frac{1}{4}}} & \lesssim \|u\|_{H^{s+\half-2\epsilon,\frac{3}{4}+\epsilon}} \|v\|_{H^{s-\half-2\epsilon,\frac{3}{4}+\epsilon}} \, , \\
\|uv\|_{H^{s-\half+2\epsilon,-\frac{1}{4}+2\epsilon}} & \lesssim \|u\|_{H^{s+\half-2\epsilon,\frac{1}{4}+3\epsilon}} \|v\|_{H^{s-\half-2\epsilon,\frac{3}{4}+\epsilon}} \, , \\
\|uv\|_{H^{s-\half+2\epsilon,-\frac{1}{4}+2\epsilon}} & \lesssim \|u\|_{H^{s+\half-2\epsilon,\frac{3}{4}+\epsilon}} \|v\|_{H^{s-\half-2\epsilon,\frac{1}{4}+3\epsilon}} \, .
\end{align*}
We apply Proposition \ref{Prop.1.2'}. It is easy to check that the first and second estimate require the assumption $s > \frac{3}{4}$ and the third estimate the assumption $s > \half$ .
\end{proof}

\begin{proof}[Proof of (\ref{29})]
It is sufficient to show
\begin{equation}
\label{G1}
\| \Gamma^1(u,v) \|_{H^{s-1,-\frac{1}{4}+2\epsilon}} \lesssim \|u\|_{H^{s,\frac{3}{4}+\epsilon}} \|v\|_{H^{s-1,\frac{3}{4}+\epsilon}}  \, .
\end{equation}
We use Lemma \ref{Lemma2.1}.\\
a. We first consider $\Gamma_2^1(u,v)$ . By (\ref{44}) it suffices to show the following estimates, all of which are consequences of Proposition \ref{Prop.1.2'}. 
\begin{align*}
\|uv\|_{H^{s-\half-2\epsilon,\frac{1}{4}}} & \lesssim \|u\|_{H^{s+\half-2\epsilon,\frac{3}{4}+\epsilon}} \|v\|_{H^{s-\half-2\epsilon,\frac{3}{4}+\epsilon}} \, ,\\
\|uv\|_{H^{s-1,-\frac{1}{4}+2\epsilon}} & \lesssim \|u\|_{H^{s,\frac{1}{4}}} \|v\|_{H^{s-\half-2\epsilon,\frac{3}{4}+\epsilon}} \, , \\
\|uv\|_{H^{s-1,-\frac{1}{4}+2\epsilon}} & \lesssim \|u\|_{H^{s+\half-2\epsilon,\frac{3}{4}+\epsilon}} \|v\|_{H^{s-1,\frac{1}{4}}} \, . 
\end{align*}
b. Assume that $u$ and $v$ have frequencies $ \ge 1$ ,  so that $\Lambda_+^{\alpha} u \sim D_+^{\alpha}u$ . In this case we use (\ref{45'}) and consider $\Gamma^1_1(u,v)$ . By (\ref{43}) we reduce to 
\begin{align*}
\|uv\|_{H^{s-\half-2\epsilon,\frac{1}{4}}} & \lesssim \|u\|_{H^{s+\half-2\epsilon,\frac{3}{4}+\epsilon}} 
\|v\|_{H^{s-\half-2\epsilon,\frac{3}{4}+\epsilon}} \, , \\
 \|uv\|_{H^{s-\half-2\epsilon,-\frac{1}{4}+2\epsilon}}& \lesssim \| u \|_{H^{s+\half-2\epsilon,\frac{1}{4}+\epsilon}} \|v \|_{H^{s-\half-2\epsilon,\frac{3}{4}+\epsilon}} \, , \\
\|uv\|_{H^{s-\half-2\epsilon,-\frac{1}{4}+2\epsilon}} & \lesssim \|u\|_{H^{s+\half-2\epsilon,\frac{3}{4}+\epsilon}} 
\|v\|_{H^{s-\half-2\epsilon,\frac{1}{4}}} \, .
\end{align*}
These estimates follow from Proposition \ref{Prop.1.2'}, which requires $s > \frac{3}{4}$. \\  
c. Consider $(\Lambda^{-2}u) v$ and $u (\Lambda^{-2}v)$ . It suffices to show
\begin{align*}
\| (\Lambda^{-2}  u)v\|_{H^{s-1,-\frac{1}{4}+2\epsilon}} &\lesssim \|u\|_{H^{s,\frac{3}{4} + \epsilon}} \|v\|_{H^{s-1,\frac{3}{4}+\epsilon}} \, , \\
\|   u(\Lambda^{-2}v)\|_{H^{s-1,-\frac{1}{4}+2\epsilon}} &\lesssim \|u\|_{H^{s,\frac{3}{4} + \epsilon}} \|v\|_{H^{s-1,\frac{3}{4}+\epsilon}}\, .
\end{align*}
Both follow easily from Prop. \ref{SML} under our asumption $s > \half$ . \\
d. Let us now consider the case where the frequencies of $u$ or $v$ are $\le 1$. We use (\ref{45''}) instead of (\ref{45'}). Because $\Gamma^1_2(u,v)$ has already been handled, we only have to consider $uv$ .
If $u$ has low frequencies we obtain by Prop. \ref{SML}:
\begin{align*}
\| uv\|_{H^{s-1,-\frac{1}{4}+2\epsilon}} & \lesssim \|u\|_{H^{\frac{3}{2}+,\frac{3}{4}+\epsilon}} \|v\|_{H^{s-1,\frac{3}{4}+\epsilon}} \\
& \lesssim \|u\|_{H^{s,\frac{3}{4}+\epsilon}} \|v\|_{H^{s-1,\frac{3}{4}+\epsilon}} \, .
\end{align*}
Similarly we treat the case where $v$ has low frequencies. 
\end{proof}

\begin{proof}[Proof of (\ref{26})]
We may reduce to
$$ \|\Lambda_+^{-1} \Lambda_-^{\epsilon -1} \Lambda_+ Q(u,v)\|_{H^{r,\half+\epsilon}} \lesssim \|\Lambda u \|_{H^{s,\frac{3}{4}+\epsilon}} \|\Lambda_+ v\|_{H^{r-1,\half+\epsilon}} \, . $$ 
Now we use (\ref{42'})  and estimate the three terms as follows: \\
1. The estimate for the first term is reduced to (using the trivial estimate $\Lambda_-^{2\epsilon} u \precsim  \Lambda_+^{2\epsilon } u $) :
$$ \|uv\|_{H^{r-\half+2\epsilon,0}} \lesssim \|u\|_{H^{s+\half-2\epsilon,\frac{3}{4}+\epsilon}} \| v \|_{H^{r-\half+2\epsilon,\half+\epsilon}} \, , $$
which follows from Prop. \ref{Prop.1.2'}. \\
2. The estimate for the last term reduces to
$$ \|uv\|_{H^{r-\half+2\epsilon,-\half+2\epsilon}} \lesssim \|u\|_{H^{s+\half-2\epsilon,\frac{3}{4}+\epsilon}} \|v\|_{H^{r-\half +2\epsilon,0}} \, , $$
which is also a consequence of Prop. \ref{Prop.1.2'} under our assumption $2s-r > \frac{3}{2}$ . \\
3. The second term reduces to
$$ \|uv\|_{H^{r-\half+2\epsilon,-\half+2\epsilon}} \lesssim \|u\|_{H^{s+\half-2\epsilon,\frac{1}{4}}}    \|v\|_{H^{r-\half+2\epsilon,\half+2 \epsilon}} \, , $$ 
which requires $s > \frac{3}{4}$ .
\end{proof}

\begin{proof}[Proof of (\ref{30}) and (\ref{31})]
 We need the following estimate
$$
\|uv\|_{H^{s-1,-\frac{1}{4}+\epsilon}}  \lesssim \|u\|_{H^{s,\frac{3}{4}+\epsilon}} \|v\|_{H^{s+1,\frac{3}{4}+\epsilon}} \, ,$$
which easily follows from the Sobolev multiplication law (Prop. \ref{SML}), because $(1-s)+s+(s+1) = s+2 > \frac{3}{2}$ , and (\ref{31}) is treated in the same way. 
\end{proof}

\begin{proof}[Proof of (\ref{33}) and (\ref{34})]
(\ref{33}) reduces to the following estimate
$$ \|uv\|_{H^{r-1,-\half+\epsilon}} \lesssim \|u\|_{H^{s+2,\frac{3}{4}+\epsilon}} \|v\|_{H^{r-1,\half+\epsilon}} \, , $$
which follows from Prop. \ref{SML}, because $s+2 > \frac{3}{2}$ . Similarly (\ref{34}) reduces to
$$ \|uv\|_{H^{r-1,-\half+\epsilon}} \lesssim \|u\|_{H^{s+1,\frac{3}{4}+\epsilon}} \|v\|_{H^{s-1,\frac{3}{4}+\epsilon}} \, , $$
which also holds by Sobolev, where we use our assumption $2s-r+1 > \frac{5}{2}$ .
\end{proof}

\begin{proof}[Proof of (\ref{35})]
We reduce to
$$ \|uvw\|_{H^{s-1,-\frac{1}{4}+2\epsilon}} \lesssim \|u\|_{H^{r+1,\half+\epsilon}} \|vw\|_{H^{r,0}}  \lesssim \|u\|_{H^{r+1,\half+\epsilon}} \|v\|_{H^{s,\frac{3}{4}+\epsilon}} \|w\|_{H^{s,\frac{3}{4}+\epsilon}} \, , $$
which follows from Prop. \ref{Prop.1.2'}, where we use $2r-s > -1$ and $r \ge s-1$ for the first step and $2s-r > \frac{3}{2}$ for the second step.
\end{proof}

\begin{proof}[Proof of (\ref{36})]
We may reduce to
\begin{align*}
 \| u \Lambda^{-1} (vw)\|_{H^{s-1,-\frac{1}{4}+2\epsilon}} & \lesssim \|u\|_{H^{r,\half+\epsilon}} \|\Lambda^{-1} (vw)\|_{H^{s+\half,0}} \\
& \lesssim  \|u\|_{H^{r,\half+\epsilon}} \|v\|_{H^{s,\frac{3}{4}+\epsilon}} \|w\|_{H^{s,\frac{3}{4}+\epsilon}} \, ,
\end{align*}
by our assumption $2r-s > -1$ ,
where we use Prop. \ref{Prop.1.2'} twice .
\end{proof}

\begin{proof}[Proof of (\ref{38})]
 It suffices to consider the case $s=\frac{3}{4}$ . We easily obtain the desired estimate by Prop. \ref{Prop.1.2'}:
$$
 \|uvw\|_{H^{-\frac{1}{4},-\frac{1}{4}+2\epsilon}}  \lesssim \|u\|_{H^{\frac{3}{4},\frac{3}{4}+\epsilon}} \|vw\|_{H^{\frac{1}{4},0}} 
\lesssim \|u\|_{H^{\frac{3}{4},\frac{3}{4}+\epsilon}}  \|v\|_{H^{\frac{3}{4},\frac{3}{4}+\epsilon}}  \|w\|_{H^{\frac{3}{4},\frac{3}{4}+\epsilon}} \, . 
$$
\end{proof}

\begin{proof}[Proof of (\ref{39})]
We obtain by Prop. \ref{Prop.1.2'}:
$$ \|uvw\|_{H^{r-1,-\half+2\epsilon}} \lesssim \|uv\|_{H^{0+,0}}\|w\|_{H^{r,\half+\epsilon}} \lesssim \|u\|_{H^{s,\frac{3}{4}+\epsilon}} \|v\|_{H^{s,\frac{3}{4}+\epsilon}} \|w\|_{H^{r,\half+\epsilon}} \, . $$
\end{proof}

\begin{proof}[Proof of (\ref{37})]
We have to show
$$ \| \Lambda^{-1} (uv) wz \|_{H^{s-1,-\frac{1}{4}+2\epsilon}} \lesssim \|u\|_{H^{s,\frac{3}{4}+\epsilon}} 
\|v\|_{H^{s,\frac{3}{4}+\epsilon}} \|w\|_{H^{s,\frac{3}{4}+\epsilon}} \|z\|_{H^{s,\frac{3}{4}+\epsilon}} \, . $$
It suffices to consider the minimal value $s=\frac{3}{4}$ , which by Proposition \ref{Prop.1.2'}  can be estimated as follows:
\begin{align*}
&\| \Lambda^{-1} (uv) wz \|_{H^{-\frac{1}{4},-\frac{1}{4}+2\epsilon}}  \lesssim \| \Lambda^{-1} (uv)  \|_{H^{\frac{3}{4}+,\half+\epsilon}} \| wz  \|_{H^{0,0}} \\
& \hspace{1em}\lesssim \| u \|_{H^{\frac{3}{4},\frac{3}{4}+\epsilon}} \|v  \|_{H^{\frac{3}{4},\frac{3}{4}+\epsilon}} \| w\|_{H^{\frac{3}{4},\frac{3}{4}+\epsilon}}  \|z  \|_{H^{\frac{3}{4},\frac{3}{4}+\epsilon}} \, . 
\end{align*}
\end{proof}

\begin{proof}[Proof of (\ref{40})]
We have to show
$$ \|uv wz \|_{H^{r-1,-\half+2\epsilon}} \lesssim \|u\|_{H^{s,\frac{3}{4}+\epsilon}} 
\|v\|_{H^{s,\frac{3}{4}+\epsilon}} \|w\|_{H^{s,\frac{3}{4}+\epsilon}} \|z\|_{H^{s,\frac{3}{4}+\epsilon}} \, . $$
By our assumption $2s-r > \frac{3}{2}$ the left hand side is bounded by the term \\
$\|uv wz \|_{H^{2s-\frac{5}{2}-\epsilon,-\half+2\epsilon}}$ . It suffices to prove the remaining estimate for the (minimal) value $s= \frac{3}{4}$ . By Proposition \ref{Prop.1.2'} we obtain
\begin{align*}
\|uv wz \|_{H^{-1,-\half+2\epsilon}} & \lesssim \|uv\|_{H^{-\frac{1}{4},\half+\epsilon}} \| wz \|_{H^{\frac{3}{8},0}} \\
& \lesssim \|u\|_{H^{\frac{3}{4},\frac{3}{4}+\epsilon}}    \|v\|_{H^{\frac{3}{4},\frac{3}{4}+\epsilon}} \| w\|_{H^{\frac{3}{4},\frac{3}{4}+\epsilon}}  \|z \|_{H^{\frac{3}{4},\frac{3}{4}+\epsilon}} \, .
\end{align*}
\end{proof}

\section{Appendix: Proof of Proposition \ref{Prop.3.8}}

The proof relies on the following fundamental result by Foschi and Klainerman.
\begin{prop}
\label{FK}
Let $n \ge 2$ . Assume $s_0+s_1+s_2+b_0 = \frac{n-1}{2}$ , $ b_0 < \frac{n-3}{4}$ , $s_0 < \frac{n-1}{2}$ , $s_1,s_2 \le \frac{n-1}{2}-b_0$ , $s_1+s_2 \ge \half$ . Then the following estimate holds:
$$ \|D_-^{-b_0} (uv) \|_{\dot{H}^{-s_0,0}} \lesssim \|u\|_{\dot{H}^{s_1,\half+\epsilon}} \|v\|_{\dot{H}^{s_2,\half+\epsilon}} \, . $$
\end{prop}
\begin{proof} 
This immediately follows from \cite{FK}, Theorem 1.1 by the transfer principle.
\end{proof}

\begin{proof}[Proof of Proposition \ref{Prop.3.8}]
We first prove (\ref{a}) in the case $s_0+s_1+s_2 =\frac{n}{2} + \epsilon$ .
By Proposition  \ref{FK} we obtain 
$$ \|D_-^{\half+\epsilon} (uv) \|_{\dot{H}^{-s_0,0}} \lesssim \|u\|_{\dot{H}^{s_1,\half+\epsilon}} \|v\|_{\dot{H}^{s_2,\half+\epsilon}} \, . $$

In the case $s_0,s_1,s_2 \ge 0$ this immediately implies
$$ \|D_-^{\half+\epsilon} (uv) \|_{H^{-s_0,0}} \lesssim \|u\|_{{H}^{s_1,\half+\epsilon}} \|v\|_{{H}^{s_2,\half+\epsilon}} \, . $$
Combining this with the estimate 
$$ \|uv\|_{H^{-s_0,0}} \lesssim \|u\|_{H^{s_1,\half+\epsilon}} \|v\|_{H^{s_2,\half+\epsilon}} \, , $$
which holds by Proposition \ref{SML} we obtain (\ref{a}).

Next consider the case $s_0 < 0$ . The estimate
$$ \|D_-^{\half+\epsilon} ((D^{-s_0}u)v)\|_{H^{0,0}} \lesssim \|u\|_{\dot{H}^{s_1,\half+\epsilon}} \|v\|_{\dot{H}^{s_2,\half+\epsilon}} $$ 
is equivalent to 
$$ \|D_-^{\half+\epsilon} (uv)\|_{H^{0,0}} \lesssim \|u\|_{\dot{H}^{s_1+s_0,\half+\epsilon}} \|v\|_{\dot{H}^{s_2,\half+\epsilon}} \, ,$$ 
which holds by Proposition \ref{FK}. Using $s_1+s_0 \ge 0$ , $s_2\ge s_0+s_2\ge 0$ this implies
$$ \|D_-^{\half+\epsilon} (uv)\|_{H^{0,0}} \lesssim \|u\|_{{H}^{s_1+s_0,\half+\epsilon}} \|v\|_{{H}^{s_2,\half+\epsilon}} \, .$$ 
By Proposition \ref{SML} we obtain
$$ \|uv\|_{H^{0,0}} \lesssim \|u\|_{H^{s_1+s_0,\half+\epsilon}} \|v\|_{H^{s_2,\half+\epsilon}} \, , $$
so that
$$ \|uv\|_{H^{0,\half+\epsilon}} \lesssim  \|u\|_{H^{s_1+s_0,\half+\epsilon}} \|v\|_{H^{s_2,\half+\epsilon}}$$
and therefore
$$ \|(\Lambda^{-s_0} u) v\|_{H^{0,\half+\epsilon}} \lesssim \|u\|_{H^{s_1,\half+\epsilon}} \|v\|_{H^{s_2,\half+\epsilon}} \, . $$
Similarly 
$$ \|u(\Lambda^{-s_0} v) \|_{H^{0,\half+\epsilon}} \lesssim \|u\|_{H^{s_1,\half+\epsilon}} \|v\|_{H^{s_2,\half+\epsilon}} \, , $$
so that (\ref{a}) follows by the fractional Leibniz rule.

Next consider the case $s_1 < 0$ (in the same way the case $s_2 < 0$ can be treated). The estimate
$$ \| (D^{-s_1} v) w \|_{H^{0,-\half-\epsilon}} \lesssim \|D_-^{-\half-\epsilon} v\|_{\dot{H}^{s_0,0}} \|w\|_{\dot{H}^{s_2,\half+\epsilon}} $$
is equivalent to
$$ \|  v w \|_{H^{0,-\half-\epsilon}} \lesssim \|D_-^{-\half-\epsilon} v\|_{\dot{H}^{s_1+s_0,0}} \|w\|_{\dot{H}^{s_2,\half+\epsilon}} \, .$$ By duality this is equivalent to
$$ \| D_-^{\half+\epsilon}(uv)\|_{\dot{H}^{-s_1-s_0,0}} \lesssim \|u\|_{H^{0,\half+\epsilon}} \|v\|_{\dot{H}^{s_2,\half+\epsilon}} \, , $$
which holds by Proposition \ref{FK}. Using $s_1+s_0 \ge 0$ and $s_2 \ge 0$ this immediately implies
$$ \| D_-^{\half+\epsilon}(uv)\|_{H^{-s_1-s_0,0}} \lesssim \|u\|_{H^{0,\half+\epsilon}} \|v\|_{{H}^{s_2,\half+\epsilon}} \, . $$
Because  the estimate
$$ \|uv\|_{H^{-s_1-s_0,0}} \lesssim \|u\|_{H^{0,\half+\epsilon}} \|v\|_{H^{s_2,\half+\epsilon}} $$
holds by Proposition \ref{SML} we obtain
$$ \|uv\|_{H^{-s_1-s_0,\half+\epsilon}} \lesssim \|u\|_{H^{0,\half+\epsilon}} \|v\|_{H^{s_2,\half+\epsilon}} \, .$$
By duality this is equivalent to
$$ \|vw\|_{H^{0,-\half-\epsilon}} \lesssim \|v\|_{H^{s_1+s_0,-\half-\epsilon}} \|w\|_{H^{s_2,\half+\epsilon}} \, , $$
thus
\begin{equation}
\label{l1}
 \|(\Lambda^{-s_1} v)w\|_{H^{0,-\half-\epsilon}} \lesssim \|v\|_{H^{s_0,-\half-\epsilon}} \|w\|_{H^{s_2,\half+\epsilon}} \, ,
 \end{equation}
Similarly the estimate
$$ \| v(D^{-s_1} w)  \|_{H^{0,-\half-\epsilon}} \lesssim \|D_-^{-\half-\epsilon} v\|_{\dot{H}^{s_0,0}} \|w\|_{\dot{H}^{s_2,\half+\epsilon}} $$
is equivalent to
$$ \|  v w \|_{H^{0,-\half-\epsilon}} \lesssim \|D_-^{-\half-\epsilon} v\|_{\dot{H}^{s_0,0}} \|w\|_{\dot{H}^{s_1+s_2,\half+\epsilon}} \, .$$
By duality this is equivalent to
$$ \|D_-^{\half+\epsilon}(uv)\|_{\dot{H}^{-s_0,0}} \lesssim \|u\|_{H^{0,\half+\epsilon}} \|v\|_{\dot{H}^{s_1+s_2,\half+\epsilon}} \, , $$
which is valid by Proposition \ref{FK}. Using $s_0 \ge 0$ and $s_1+s_2 \ge 0$ this immediately implies
$$ \|D_-^{\half+\epsilon}(uv)\|_{{H}^{-s_0,0}} \lesssim \|u\|_{H^{0,\half+\epsilon}} \|v\|_{{H}^{s_1+s_2,\half+\epsilon}} \, , $$
which combined with Proposition \ref{SML}, which gives the estimate
$$\|uv\|_{H^{-s_0,0}} \lesssim \|u\|_{H^{0,\half+\epsilon}} \|v\|_{H^{s_1+s_2,\half+\epsilon}} \, , $$
implies
$$ \|uv\|_{H^{-s_0,\half+\epsilon}} \lesssim \|u\|_{H^{0,\half+\epsilon}} \|v\|_{H^{s_1+s_2,\half+\epsilon}} \, . $$
By duality this is equivalent to
$$ \|vw\|_{H^{0,-\half-\epsilon}} \lesssim \|v\|_{H^{s_0,-\half-\epsilon}} \|w\|_{H^{s_1+s_2,\half+\epsilon}} $$
and also
\begin{equation}
\label{l2}
\|v \Lambda^{-s_1}w\|_{H^{0,-\half-\epsilon}} \lesssim \|v\|_{H^{s_0,-\half-\epsilon}} \|w\|_{H^{s_2,\half+\epsilon}} \, .
\end{equation}
(\ref{l1}) and (\ref{l2}) imply by the fractional Leibniz rule
$$\|vw\|_{H^{-s_1,-\half-\epsilon}} \lesssim \|v\|_{H^{s_0,-\half-\epsilon}} \|w\|_{H^{s_2,\half+\epsilon}} \, . $$
This is by duality our desired estimate (\ref{a}).

Consider now the more general case $s_0+s_1+s_2 \ge \frac{n}{2}+\epsilon$ . It is not difficult to see that in the case of (homogeneous) $H^{s,b}$-spaces the inequality (\ref{a}) may be reduced to the special case considered before.

Moreover a simple application of Proposition \ref{SML} shows that
\begin{equation}
\label{d}
\|uv\|_{H^{-s_0,0}} \lesssim \|u\|_{H^{s_1,0}} \|v\|_{H^{s_2,\half+\epsilon}} \, .
\end{equation}
Estimate (\ref{b}) now follows by interpolation between  (\ref{a}) and (\ref{d}).

\end{proof}


\begin{thebibliography}{999999}
\bibitem[AFS]{AFS}
P. d'Ancona, D. Foschi, and S. Selberg: 
{\sl Atlas of products for wave-{S}obolev spaces on
  {$\mathbb{R}^{1+3}$}}.
  Trans. Amer. Math. Soc. 364, (2012), 31-63.
\bibitem[AFS1]{AFS1}
P. d'Ancona, D. Foschi, and S. Selberg: {\sl Null structure and
  almost optimal local well-posedness of the {M}axwell-{D}irac system}. Amer.
  J. Math. 132 (2010), 771--839. 
\bibitem[FK]{FK} D. Foschi and S. Klainerman:  {\sl Bilinear space-time estimates for homogeneous wave equations}.  Ann. Scient. Ec. Norm. Sup. $4^e$ ser. , 33 (2000), 211-274
\bibitem[GV]{GV} J. Ginibre and G. Velo: {\sl Generalized Strichartz inequalities for the wave equation.} J. Functional Anal. 133 (1995), 60-68
\bibitem[KM]{KM}  S. Klainerman and M. Machedon: {\sl Finite energy solutions of the Yang-Mills equations in ${\mathbb R}^{3+1}$}. Ann. Math. 142 (1995), 39-119
\bibitem[KMBT]{KMBT} S. Klainerman and M. Machedon (Appendices by J. Bougain and  D. Tataru): {\sl Remark on Strichartz-type inequalities}. Int. Math. Res. Notices 1996, no.5, 201-220 
\bibitem[KS]{KS} S. Klainerman and S. Selberg: {\sl  Bilinear estimates and applications to nonlinear wave equations}. Communications in Contemporary Mathematics. 4 (2002) 223-295. 
\bibitem[KT]{KT} S. Klainerman and D. Tataru: {\sl On the optimal local regularity for the Yang-Mills equations in $\mathbb{R}^{4+1}$}. Journal of the AMS 12 (1999), 93-116
\bibitem[KrSt]{KrSt} J. Krieger and J. Sterbenz: {\sl Global regularity for the Yang-Mills equations on high dimensonal Minkowski space}. Mem. AMS 223 (2013), No. 1047
\bibitem[KrT]{KrT} J. Krieger and D. Tataru: {\sl Global well-posedness for the Yang-Mills equations in 4+1 dimensions. Small energy.}  Ann. of Math. 185 (2017), 831–893.
\bibitem[LV]{LV} S. Lee and A. Vargas: {\sl Sharp null form estimates for the wave equation}. Amer. J. Math. 130 (2008), 1279-1326 
\bibitem[O]{O} S. Oh: {\sl Gauge choice for the Yang-Mills equations using the Yang-Mills heat flow and local well-posedness in $H^1$}. J. Hyperbolic Differ. Equ. 11 (2014), 1–108. 
\bibitem[O1]{O1} S. Oh: {\sl  Finite energy global well-posedness of the Yang-Mills equations on $\mathbb{R}^{1+3}$: an approach using the Yang-Mills heat flow.} Duke Math. J. 164 (2015), 1669–1732.
\bibitem[P]{P} H. Pecher: {\sl Local well-posedness for the (n+1)-dimensional Yang-Mills and Yang-Mills-Higgs system in temporal gauge}. Nonl. Diff. Equ. Appl. 23 (2016), 23-40
\bibitem[ST1]{ST1}
S. Selberg and A. Tesfahun, {\sl Finite-energy global well-posedness
  of the {M}axwell-{K}lein-{G}ordon system in {L}orenz gauge}. Comm. Partial
  Differential Equations 35 (2010), 1029--1057. 
\bibitem[ST]{ST} S. Selberg and A. Tesfahun: {\sl Null structure and local well-posedness in the energy class for the Yang-Mills equations in Lorenz gauge}. Journal of the European Mathematical Society 18 (2016), 1729–1752.
\bibitem[St]{St}  J. Sterbenz: {\sl Global regularity and scattering for general non-linear wave equations. II. (4+1) dimensional Yang-Mills equations in the Lorentz gauge.} Amer. J. Math. 129 (2007), 611–664. 
\bibitem[T]{T1} T. Tao: {\sl Local well-posedness of the Yang-Mills equation in the temporal gauge below the energy norm}. J. Diff. Equ. 189 (2003), 366-382
\bibitem[Te]{Te} A. Tesfahun: {\sl Local well-posedness of Yang-Mills equations in Lorenz gauge below the energy norm}. Nonlin. Diff. Equ. Appl. 22 (2015), 849-875
\end{thebibliography}
\end{document}